\documentclass[12pt, reqno]{amsart}
\usepackage{amsthm}
\usepackage{amssymb}
\usepackage{amsmath}

\def\1{\hbox{1\kern-.35em\hbox{1}}}

\newtheorem{claim}{\indent Claim}
\newtheorem{theorem}{Theorem}[section]
\newtheorem*{theorem*}{Theorem}
\newtheorem{lemma}[theorem]{Lemma}
\newtheorem{proposition}[theorem]{Proposition}
\newtheorem*{proposition*}{Proposition}

\newtheorem{definition}[theorem]{Definition}
\newtheorem{notation}[theorem]{Notation}
\newtheorem{remark}[theorem]{Remark}

\newtheorem{example}[theorem]{Example}
\newtheorem{convention}[theorem]{Convention}

\numberwithin{equation}{section}

\newcommand{\bea}{\begin{eqnarray}}
\newcommand{\eea}{\end{eqnarray}}
\newcommand{\be}{\begin{eqnarray*}}
\newcommand{\ee}{\end{eqnarray*}}

\newcommand{\Z}{{\mathbb Z}}

\newcommand{\C}{{\mathbb C}}

\newcommand{\fg}{{\mathfrak g}}
\newcommand{\fh}{{\mathfrak h}}

\def\UPL#1#2{\put(-5,#1){\line(0,1){5}\put(0,5){\line(1,0){#2}}}}
\def\UPR#1#2{\put(-5,#1){\line(0,1){5}\put(0,5){\line(-1,0){#2}}}}
\def\DNL#1#2{\put(-5,-#1){\line(0,1){5}\line(1,0){#2}}}
\def\DNR#1#2{\put(-5,-#1){\line(0,1){5}\line(-1,0){#2}}}

\def\bigskip{\vskip6pt}
\def\medskip{\vskip6pt}

\def\ker#1{{\rm Ker}(#1)}
\def\Im#1{{\rm Im}(#1)}

\def\a{\alpha}

\def\sign{{\rm sign}}
\def\b{\beta}
\def\d{\delta}
\def\D{\Delta}
\def\g{\gamma}
\def\G{\Gamma}
\def\th{\theta}
\def\l{\lambda}
\def\L{\Lambda}

\def\Si{\Sigma}
\def\si{\sigma}
\def\sc{\scriptstyle}
\def\ssc{\scriptscriptstyle}
\def\dis{\displaystyle}

\def\wt{\widetilde}

\def\rar{\rightarrow}

\def\rrar{\rightarrow}
\def\llar{\leftarrow}

\def\drar{\mbox{${\sc\,}\rrar\!\!\!\!\rar{\sc\,}$}}

\def\dlar{\mbox{${\sc\,}\lar\!\!\!\!\llar{\sc\,}$}}

\def\link{\mbox{${\ssc\,}\lar\!\!\!\!\!\!-\!\!\!\!\!\!\rar{\ssc\,}$}}

\def\lar{\leftarrow}
\def\D{\Delta}

\def\Lra{\Longleftrightarrow}
\def\bs{\backslash}

\def\rb{\raisebox}
\def\ch{{\rm ch{\ssc\,}}}

\def\N{\mathbb{N}}
\def\Z{\mathbb{Z}}

\def\C{\mathbb{C}}

\def\gl{{\mathfrak g}}

\def\vp{\varphi}
\def\es{\varepsilon}

\def\refequa#1{(\ref{#1})}
\def\stl{\stackrel}

\def\ops{{\mathfrak{osp}}_{k|2}}

\def\nex{{{}\hat{}}}
\def\pre{{{}\check{}}}

\def\DEP{{\D_{\bar0}^+}}
\def\DOP{{\D_{1}^+}}
\def\equa#1#2{
\begin{equation}\label{#1}\mbox{$#2$}\end{equation}}
\def\equan#1#2{$$\mbox{$#2$}$$}

\numberwithin{equation}{section}

\def\cO{\mathcal{O}}

\begin{document}

\title[Generalised Verma modules for Lie superalgebras]
{Generalised Verma modules for \\ the orthosymplectic Lie
superalgebra ${\mathfrak{osp}}_{k|2}$}

\author[Yucai Su]{Yucai Su}
\address{Yucai SU,
Department of Mathematics, University of Science and Technology of
China, Hefei 230026, China.} \email{ycsu@ustc.edu.cn}
\author[R. B. Zhang]{R. B. Zhang}
\address{R. B. Zhang, School of Mathematics and Statistics,
University of Sydney, NSW 2006, Australia.}
\email{rzhang@maths.usyd.edu.au}

\begin{abstract}
The composition factors and their multiplicities are determined for
generalised Verma modules over the orthosymplectic Lie superalgebra
${\mathfrak{osp}}_{k|2}$. The results enable us to obtain explicit
formulae for the formal characters and dimensions of the
finite-dimensional irreducible modules. Applying these results, we
also compute the first and second cohomology groups of the Lie
superalgebra with coefficients in finite-dimensional Kac modules and
irreducible modules.

\end{abstract}
\maketitle

\tableofcontents

%
\section{Introduction}\label{Introduction}

Lie superalgebras continue to attract much interest in both
mathematics and physics. These algebras originated from quantum
field theory in the 1970s, and provide the mathematical foundation
for the physical notion of supersymmetry. [The Large Hadron Collider
built by the European Organization for Nuclear Research is currently
in operation to verify supersymmetry experimentally.]  It was
already clear from the foundational works of Kac \cite{K77, Kac2,
Kac3} that the representation theory of finite-dimensional simple
Lie superalgebras was very different from that of the ordinary
simple Lie algebras. In particular, finite-dimensional
representations of such a Lie superalgebra over $\C$ are not
semi-simple in general, in sharp contrast to those of the latter.
This renders the development of the representation theory of Lie
superalgebras a more interesting but harder task.

The finite-dimensional irreducible representations of some Lie
superalgebras \cite{V85, V91} and classes of representations of the
general linear superalgebra \cite{DJ, BR, BL, PS} were understood
long ago, and very precise conjectures were also formulated on the
representation theory of the general linear superalgebra
\cite{VHKT}.  However, it was within the last fifteens year that
major advances were achieved. Serganova in \cite{Se96, Se98} further
developed the geometric approach \cite{P, PS} to the representation
theory of Lie superalgebras  to derive algorithms for computing
characters of finite-dimensional irreducible representations.
Brundan \cite{B} discovered a remarkable connection between the
general linear superalgebra and the quantum group of
${\mathfrak{gl}}_\infty$. He reformulated the generalised
Kazhdan-Lusztig theory for the Lie superalgebra in terms of
canonical bases of the quantum group, proving the algorithm for
determining the composition numbers of the Kac modules conjectured
in \cite{VZ}. The approach of \cite{B} was further developed in
\cite{Su, SuZ2, CWZ, CL}. In particular, a closed character formula
was obtained for the finite-dimensional irreducible representations
of the general linear superalgebra in \cite{SuZ2}, and the
generalised Kazhdan-Lusztig polynomials of ${\mathfrak{gl}}(m|N)$
were shown to be equal to those of ${\mathfrak{gl}}(m+N)$ (when
$N\to\infty$) in \cite{CWZ} (and earlier in \cite{CZ} for polynomial
representations). This suggested an equivalence \cite[Conjecture
6.10]{CWZ} between the appropriate parabolic category $\cO$ of
${\mathfrak{gl}}(m|\infty)$ and that of ${\mathfrak{gl}}(m+\infty)$,
which was recently established in \cite{CL} and \cite{BS}
independently using very different methods.

Very recently, Gruson and Serganova \cite{GS} improved upon earlier
work of Serganova \cite{Se98} to give combinatorial algorithms for
computing characters of finite-dimensional irreducible
representations of $\mathfrak{osp}_{k|2n}$. An interesting fact is
that references \cite{Se96, Se98} worked entirely within the
category of finite-dimensional $\mathfrak{osp}_{m|2n}$-modules,
where certain virtual modules built from cohomology groups of
locally free equivariant sheavs over super Grassmannians played the
role of Kac modules in the case of type $A$. In this paper we
further study the representation theory of the orthosymplectic Lie
superalgebras.

In this paper we study the structure of the generalised Verma
modules for the orthosymplectic Lie superalgebras. The goal here is
somewhat different from that of references \cite{Se96, Se98}, as
generalised Verma modules do not play a role within the framework of
these papers. We shall focus on $\mathfrak{osp}_{k|2n}$ for $n=1$
and $k\ge 3$. We choose the distinguished Borel subalgebra for
$\mathfrak{osp}_{k|2}$ and consider the maximal parabolic subalgebra
obtained by removing the unique odd simple root. The generalised
Verma modules studied here are those induced from finite-dimensional
irreducible modules over the parabolic subalgebra.

One of our main results is determining the composition factors of
the generalised Verma modules for $\mathfrak{osp}_{k|2}$ and also
obtaining their multiplicities (see  Theorem
\ref{theorem-verma-module} for the precise statement of the result).
Using this result we construct explicit formulae for the characters
and dimensions of the finite-dimensional irreducible modules in
Theorem \ref{main-theo1}. The formulae are sums of a finite number
of terms, each of which resembles Weyl's character or dimension
formula for semi-simple Lie algebras. We point out that a character
formula (in the form of an infinite sum) for the finite dimensional
irreducible $\mathfrak{osp}_{k|2}$-modules was obtained very
recently by Luo \cite{L} using a different method.

The maximal finite-dimensional quotients of the generalised Verma
modules with dominant highest weights are the Kac modules \cite{K77,
Kac2, Kac3} for the orthosymplectic Lie superalgebras. Kac modules
naturally arise as the zeroth cohomology groups of some locally free
equivariant sheavs over the super Grassmannian corresponding to the
parabolic subalgebra. However, if the highest weight is atypical,
the Kac module can be irreducible in some special cases but is not
irreducible in general. Therefore it is also interesting to
understand the structure of Kac modules. We determine the characters
of the Kac modules in Proposition \ref{Kac-ch}.

By using results obtained on the structure of the generalised Verma
modules and Kac modules, we compute explicitly the first and second
cohomology groups of the orthosymplectic Lie superalgebra with
coefficients in finite-dimensional irreducible modules and Kac
modules in Theorem \ref{1-coho}. This part of the paper is a
continuation of \cite{SZ98, SuZ1} on the study of Lie superalgebra
cohomology.

Closely related to our work is the very recent paper by Cheng, Lam
and Wang \cite{CLW}, which proved an equivalence between the
parabolic $\cO$ categories of ${\mathfrak{osp}}_{m|\infty}$ and
${\mathfrak{so}}_{m|\infty}$ in analogy to the type-$A$ case.
Equipped with this equivalence of categories, one in principle can
study the representation theory of ${\mathfrak{osp}}_{m|2n}$ for
finite $n$ by first understanding representations of
${\mathfrak{so}}_{m|\infty}$, then transcribing the results to the
corresponding representations of ${\mathfrak{osp}}_{m|\infty}$, and
finally deducing precise results for
${\mathfrak{osp}}_{m|2n}$-representations by truncating the
${\mathfrak{osp}}_{m|\infty}$-representations to finite $n$. Such a
method worked reasonably well in the case of the polynomial
representations of ${\mathfrak{gl}}(m|n)$ \cite{CZ}, and it is
expected to work in the case of ${\mathfrak{osp}}_{m|2n}$ as well.


%
%
\section{Preliminaries}
\label{Preliminaries}

This section contains some background material which will be needed
later.  It also serves to fix notation. We work over the field of
complex numbers throughout the paper.

\subsection{The Lie superalgebras $\ops$}
We denote by $\gl$ the orthosymplectic Lie superalgebra $\ops$ for
$k>2$. The algebra $\mathfrak{osp}_{2|2}$ is of type I \cite{K77},
whose finite-dimensional irreducible modules have long been
understood \cite{V91}. Choose the distinguished Borel subalgebra
\cite{K77,Sch} for $\gl$. Then the corresponding Dynkin diagram is
given by
\equan{Dynkin}{
\begin{array}{lllll} \gl=D(m,1):&
\stl{\d-\es_1}{{\sc\otimes}}\!\!\!\!\!-\!\!\!\!\!-\!\!\!\!\!-\!\!\!\!\!-\!\!\!\!\!
-\!\!\!\!\!-\!\!\!\!\!-\!\!\!\!\!-\!\!\!\!\!
\stl{\es_1-\es_2}{\circ}
\!\!\!\!\!-\!\!\!\!\!-\!\!\!\!\!-\!\!\!\!\!-\!\!\!\!\!-\!\!\!\!\!-\!\!\!\!\!
-\!\cdots
\!-\!\!\!\!\!-\!\!\!\!\!-\!\!\!\!\!-\!\!\!\!\!-\!\!\!\!\!-\!\!\!\!\!
-\!\!\!\!\!-\!\!\!\!\!-\!\!\!\!\!-\!\!\!\!\!-
\!\!\!\!\!\!\!\!\stl{\es_{m-2}-\es_{m-1}}{-\!\!\!\!\!
-\!\!\!\!\!-\!\!\!\!\!-\!\!\!\!\!-\!\!\!\!\!
-\!\!\!\!\!-\!\!\!\!\!-\!\!\!\!\!-\!\!\!\!\!
-\!\!\!\!\!-\!\!\!\!\!-\!\!\!\!\!-\!\circ}
\put(-8,5){\line(2,1){10}\put(-10,-4){\line(2,-1){10}}}
\put(3,7){$\circ\,\put(2,3){$\sc\es_{m-1}-\es_m$}$}
\put(3,-8){$\circ\,\sc\es_{m-1}+\es_m$}\ \hspace*{45pt}\  &\mbox{if
\ }k=2m,
\\[15pt]
\gl=B(m,1):&
\stl{\d-\es_1}{{\sc\otimes}}\!\!\!\!\!-\!\!\!\!\!-\!\!\!\!\!-\!\!\!\!\!-\!\!\!\!\!
-\!\!\!\!\!-\!\!\!\!\!-\!\!\!\!\!-\!\!\!\!\!
\stl{\es_1-\es_2}{\circ}
\!\!\!\!\!-\!\!\!\!\!-\!\!\!\!\!-\!\!\!\!\!-\!\!\!\!\!-\!\!\!\!\!-\!\!\!\!\!
-\!\cdots
\!-\!\!\!\!\!-\!\!\!\!\!-\!\!\!\!\!-\!\!\!\!\!-\!\!\!\!\!-\!\!\!\!\!
-\!\!\!\!\!-\!\!\!\!\!-\!\!\!\!\!-\!\!\!\!\!-
\!\!\!\!\!\!\!\!\!\!\!\!\!\!\!\!\stl{\ \ \ \ \ \
\es_{m-1}-\es_{m}}{-\!\!\!\!\!-\!\!\!\!\!-\!\!\!\!\!-\!\!\!\!\!-\!\!\!\!\!
-\!\!\!\!\!-\!\!\!\!\!-\!\!\!\!\!-\!\!\!\!\!-\!\!\!\!\!-\!\!\!\!\!-\!\!\!\!\!
-\!\circ\!\!\!\!\!}
\!\!\!\!\!\Longrightarrow\!\stl{\es_m}{\circ}
 \ \hspace*{0pt}\ &\mbox{if \ }k=2m+1,
\end{array}
}
where the grey node is associated with the unique odd simple
root. We shall denote by $\Pi$ the set of simple roots. We also denote by $\DEP$ and $\DOP$ the sets of
positive even roots and positive odd roots respectively, which are given by
\equa{root}
{
\begin{array}{lll}
\DEP=\{2\d,\es_i\pm\es_j\,|\,1\!\le\! i\!<\!j\!\le\!
m\},&\DOP=\{\d\pm\es_i\,|\,1\!\le\! i\!\le\! m\}&\mbox{if
}k\!=\!2m,\\[7pt]
\DEP=\{2\d,\es_i,\es_i\pm\es_j\,|\,1\!\le\! i\!<\!j\!\le\!
m\},\!\!&\DOP=\{\d,\d\pm\es_i\,|\,1\!\le\! i\!\le\! m\}\!\!&\mbox{if
}k\!=\!2m\!+\!1. \nonumber
\end{array}
}
We define the order
\equa{order-weight}{\mbox{$0<\d-\es_1<\es_1-\es_2<\cdots<\es_{m-1}-\es_m$
and $0<\d,\es_i$,} \nonumber} which gives a total order on the set
of roots.

\begin{notation}\label{e-f-a}
\rm For $\a\in\Pi$, we fix $e_\a,\,f_\a$ to be respectively  the
positive, negative root vectors. For any $\b\in\D^+\bs\Pi$, where
$\D^+=\DEP\cup\DOP$, we uniquely define the positive, negative root
vectors $e_\b,\,f_\b$ as follows: Let $\a\in\Pi$ be the unique
minimal simple root such that $\b-\a\in\D^+$, then
$e_\b=[e_\a,e_{\b-\a}],$ $f_\b=[f_{\b-\a},f_\a]$. We also write
$e_{-\b}=f_\b$ and define $h_\b=[e_\b,f_\b]$ for $\b\in\D^+$.
\end{notation}

Then $h_\a,\,\a\in\Pi$ forms a basis of the Cartan subalgebra $\fh$.
The bilinear form $(\cdot,\cdot)$ on $\fh^*$ is defined by
\equa{form}{ (\d,\d)=-1,\,\ \ (\d,\es_i)=0,\,\ \
(\es_i,\es_j)=\d_{i,j}\mbox{ \ \ for \ }1\le i, j\le m. \nonumber}

It is clear that $\gl=\oplus_{i=-2}^2\gl_i$ is a $\Z_2$-consistent
$\Z$-graded Lie superalgebra such that
$\gl_{\bar0}=\gl_{-2}\oplus\gl_0\oplus\gl_2$ and
$\gl_{\bar1}=\gl_{-1}\oplus\gl_{+1}$. Here $\gl_{\pm2}=\C
e_{\pm2\d}$ and  $\gl_0={\mathfrak{so}}(m)$ with the set of positive
roots $\D_0^+=\DEP\bs\{2\d\}$.

A weight $\l=\l_0\d+\sum_{i=1}^m\l_i\es_i\in\fh^*$ is written as
\equan{weight}{\l=(\l_0\,|\,\l_1,....,\l_m).} We shall call $\l_i$
the {\it $i$-th coordinate} of $\l$. Denote by $\rho$ the half
signed sum of positive roots, and by $\rho_1$ the half sum of
positive odd roots. Then \equa{rho}{
\begin{array}{lll}
\rho=(1-m\,|\,m-1,m-2,...,0),&\rho_1=(m\,|\,0,...,0), &\mbox{if
}k=2m,\\[7pt]
\rho=(\frac12-m\,|\,m-\frac12,m-\frac32,...,\frac12),
&\rho_1=(m+\frac12\,|\,0,...,0),
 &\mbox{if }k=2m+1.
\end{array}
}
For convenience, we always denote $s=1$ if $k=2m$ and $s=\frac12$
if $k=2m+1$. Then $\rho=(s-m\,|\,m-s,...,1-s)$ and
$\rho_1=(m+1-s\,|\,0,...,0)$. Sometimes, it will be convenient to
use the so-called {\it $\rho$-translated notation} of a weight $\l$,
which will be always denoted by the same symbol with a tilde:
\equa{rho-translated}{\tilde\l=\l+\rho=(\tilde
\l_0\,|\,\tilde\l_1,....,\tilde\l_m).} Then $\tilde\l_0=\l_0+s-m$
and $\tilde\l_i=\l_i+m+1-s-i$ for $i>0$.

Denote by $W$ and $W_0$ the Weyl groups of $\gl$ and $\gl_0$
respectively. Then $W=W_0\times\Z_2$, where the nontrivial element
$\si\in \Z_2$ changes the sign of the $0$-th coordinate when acting
on a weight. $W_0$ is the Weyl group of ${\mathfrak{so}}(m)$, which
acts on a weight $\l\in\fh^*$ by permuting the coordinates
$\l_1,...,\l_m$ and also changing their signs (the number of sign
changes must be in $2s\Z_+$). Define the {\it dot action} of $W$ on
$\fh^*$ by \equa{dot-action}{w\cdot\l=w(\tilde\l)-\rho\mbox{ \ for \
}w\in W,\,\,\l\in\fh^*.} We remark that the dot action of $W_0$ on
$\fh^*$ is defined by $w\cdot\l=w(\l+\rho_{\bar0})-\rho_{\bar0}$,
where $\rho_{\bar0}=\rho+\rho_1$ is the half sum of positive even
roots. Since $\rho_1$ is $W_0$-invariant, this dot action coincides
with (\ref{dot-action}).
 We denote
\equa{l=si}{\l^\si=\si\cdot\l=(-\l_0+2(m-s)\,|\,\l_1,...,\l_m).}

Let $V=\oplus_{\l\in\fh^*}V_\l$ be a {\em weight module} over $\fg$,
where \equan{weight-space} {V_\l=\{v\in
V\,|\,hv=\l(h)v,\,\forall\,h\in\fh\}\mbox{ \ with \ }\dim
V_\l<\infty, } is the weight space of weight $\l$. The {\it
character} $\ch V$ is defined to be \equan{char-def} {\mbox{$ \ch
V=\sum\limits_{\l\in\fh^*}({\rm dim\,}V_\l) e^\l, $}} where $e^\l$
is the {\em formal exponential}, which will be regarded as an
element of an additive group isomorphic to $\fh^*$ under $\l\mapsto
e^\l$. Then $\ch V$ is an element of the {\em completed group
algebra} \equan{epsilon} {\mbox{$
\varepsilon=\bigl\{\sum\limits_{\l\in\fh^*}a_\l
e^\l\,\bigl|\,a_\l\in\C,\, a_\l=0\mbox{ except $\l$ is in a finite
union of ${\mathcal Q}_\L$}\bigr\}, $}} where for $\L\in\fh^*$,
\equan{Q} {\mbox{$ {\mathcal Q}_\L=\bigl\{\L-\sum\limits_{\a\in\D^+}
i_\a\a\in\fh^*\,\bigl|\,i_\a\in\Z_+\bigr\}. $}} Sometimes, we may
also work with elements in the group $\bar{\es}$ which is defined as
above with ${\mathcal Q}_\L$ replaced by \equa{bar-Q} {\mbox{$
{\bar{\mathcal Q}}_\L=\bigl\{\L+\sum\limits_{\a\in\D^+}
i_\a\a\in\fh^*\,\bigl|\,i_\a\in\Z_+\bigr\}. $}}

\subsection{Generalised Verma modules and Kac modules}
A weight $\l\in\fh^*$ is called {\it integral} if
\equa{integral}
{\l_0\in \Z\mbox{ \ \  and \ }\l_i\in\Z\mbox{ or }s+\Z.}
Denote by $P$ the set of integral weights. An integral weight $\l\in P$ is
{\it regular} if $|\tilde\l_1|,...,|\tilde\l_m|$ are distinct
number; {\it $\gl_0$-dominant} if
\equa{g0-dominant}{\l_0\ge0,\
\l_1\ge...\ge\l_{m-1}\ge|\l_m|, \mbox{ \ and further,
}\l_m\ge0\mbox{ if }k=2m+1.}
We denote by $P^{0+}$ the set of
integral $\gl_0$-dominant weights. Then every regular weight $\l$ is
$W_0$-conjugate under the dot action \refequa{dot-action} to a
unique integral $\gl_0$-dominant weight, which will be denoted by
$\l^+$ throughout the paper.

For $\l\in P^{0+}$, let $L_\l^{(0)}$ be the finite-dimensional
irreducible $\gl_0$-module with highest weight $\l$. Extend it to a
$\gl_0\oplus\gl_{+1}\oplus\gl_{+2}$-module by putting
$(\gl_{+1}\oplus\gl_{+2})L^{(0)}_\l=0$. Then the {\em generalised
Verma module} $V_\l$ is the induced module \cite{Kac2, Kac3}
\equa{Verma-module}{ V_\l={\rm
Ind}_{\gl_0\oplus\gl_{+1}\oplus\gl_{+2}}^{\gl}L^{(0)}_\l\cong
U(\gl_{-1}\oplus\gl_{-2}) \otimes_{\C}L^{(0)}_\l. }
One has the following easy result.
\begin{lemma}\label{char-Verma}
If $\l$ is an integral $\fg_0$-dominant weight, then
\equa{char-Verma-for} { \begin{aligned} \ch
V_\l&={\dis\frac{R_1}{R_{\bar0}}}\, \mbox{$\sum\limits_{w\in
W_0}$}\, \sign(w)e^{w(\l+\rho)}\\
&={\dis\frac1{R_{\bar0}}}\sum\limits_{w\in
W_0}\sign(w)w\Big(e^{\l+\rho_{\bar0}}\prod\limits_{\b\in\D_1^+}(1+e^{-\b})\Big),
\end{aligned}
}
where $\sign(w)$ is the signature of $w\in W$, and
\equan{l-0}
{\mbox{$
R_{\bar0}=\prod\limits_{\a\in\D_{\bar0}^+}(e^{\a/2}-e^{-\a/2}),\,\ \
R_1=\prod\limits_{\b\in\D_1^+}(e^{\b/2} +e^{-\b/2}). $}}
\end{lemma}

Denote by $L_\l$ the unique irreducible quotient module of $V_\l$.
For convenience we shall always fix a highest weight vector $v_\l$
in $V_\l$ or $L_\l$. Obviously, for any $\l\in P^{0+}$, the
generalised Verma module $V_\l=\oplus_{i\in\Z_-} V_\l^{(i)}$ is
$\Z$-graded with respect to the eigenvalue of $\frac12 h_{2\d}$ such
that each $V_\l^{(i)}$ is a finite-dimensional $\gl_0$-module.

A weight $\l$ is called {\it integral $\fg$-\it dominant} if it
satisfies the conditions \refequa{integral}, \refequa{g0-dominant}
and \equa{d-condition}{\ell=\l_0\in\Z_+,\mbox{ \ and if }0\le\ell\le
m-1\mbox{ then }\l_{\ell+1}=\l_{\ell+2}=...=\l_m=0.} Let $P^+$
denote the set of integral $\fg$-dominant weights. It was shown in
\cite{K77} that $\dim L_\l<\infty$ if and only if $\l\in P^+$. In
this case, we can define the quotient module
\equa{Kac-module}{K_\l=V_\l/U(\gl)f_{2\d}^{\l_0+1}v_\l,} which is
usually referred to as the {\it Kac module} \cite{Kac2, Kac3}.
Obviously, by \refequa{Verma-module} and the PBW Theorem,
\equan{Kac-dim}{{\rm dim\,}K_\l\le(\l_0+1)2^{\ssc\,{\rm
dim\,}\gl_{-1}}\,{\rm dim\,}L_\l^{(0)}<\infty.}

\subsection{Atypicality and central characters}
An integral $\gl_0$-dominant weight $\l$ is called {\it atypical
with atypical root $\g=\d+\es_\ell$}  (resp. $\g=\d-\es_\ell$) if
$\tilde\l_0=\tilde\l_\ell$ (resp. $\tilde\l_0=-\tilde\l_\ell$) for
some $1\le\ell\le m$. A weight that is not atypical is called {\it
typical}.
\begin{remark}\label{l-0=0=l-m=0}\rm
In case $\tilde\l_0=\tilde\l_\ell=0$ (this can only occur when
$k=2m$), both roots $\g_\pm=\d\pm\es_\ell$ are atypical roots of
$\l$. In this case, we always choose $\g=\d-\es_\ell$.
\end{remark}
\begin{definition}\label{tail} \rm
 (cf.~\cite{Se98}) In case
$\tilde\l_0=-\tilde\l_\ell$, the atypical root $\g=\d-\es_\ell$ is
called a {\it tail atypical root} (cf.~Example \ref{nex-pre}).
\end{definition}
Assume $\l\in P^{0+}$ is atypical with atypical root
$\g=\d+\es_\ell$ or $\d-\es_\ell$. Then conditions $\l_0\in\Z$ and
$\tilde\l_0=\pm\tilde\l_\ell$ force $\l_i\in\Z$ for all $i$. Denote
\equa{atypical-type}{\bar\l=(|\tilde\l_1|,...,|\tilde\l_{\ell-1}|,
|\tilde\l_{\ell+1}|,...,|\tilde\l_m|),\ \ \
S(\bar\l)=\{|\tilde\l_1|,...,|\tilde\l_{\ell-1}|,|\tilde\l_{\ell+1}|,...,|\tilde\l_m|\}.}
(Note from \refequa{g0-dominant} that $\tilde\l_m$ is the only
possible negative number in \refequa{atypical-type}.) We call
$\bar\l$ the {\it atypicality type} of $\l$.

Denote by $Z(\gl)$ the center of the universal enveloping algebra $U(\gl)$.
An element $z\in Z(\gl)$ can be uniquely written in the form
$z=u_z+\sum_i u_i^-u_i^0u_i^+$, where $u_z,u_i^0\in
U(\fh),\,u_i^{\pm}\in\gl_{\pm}U(\gl_{\pm})$, and $\gl_{\pm}$ is the
subalgebra of $\gl$ spanned by the positive/negative root vectors.
Then the map $z\mapsto u_z$ gives rise to the {\it Harish-Chandra
homomorphism} \equa{HC}{HC:Z(\gl)\to\C[\fh^*],\ \ \
\,HC(z)(\l)=u_z(\l).} Recall that a {\it central character} is a
homomorphism $Z(\gl)\to\C$. Thus any $\l\in\fh^*$ defines a central
character $\chi_\l$ by the rule $\chi_\l(z)=HC(z)(\l)$, such that
all element $z\in Z(\gl)$ acts on $L_\l$ as the scalar $\chi_\l(z)$.
One has \cite{Se98} \equa{central-atypical}{\chi_\l=\chi_\mu\ \ \
\Longrightarrow\ \ \ \bar\l=\bar\mu\mbox{ \ \ for all atypical
}\l,\mu\in P^+.}

The following result is due to Kac \cite{Kac2, Kac3}:
\begin{proposition}\label{typical-prop}
If $\l$ is an integral $\fg$-dominant  typical weight, then
$L_\l=K_\l$, and \equa{typical-char} { \ch L_\l=\ch
K_\l={\dis\frac{R_1}{R_{\bar0}}}\, \mbox{$\sum\limits_{w\in W}$}\,
\sign(w)e^{w(\l+\rho)}. }
\end{proposition}

\section{Structure of generalised Verma modules}

In this section, we shall completely determine the structure of generalised Verma
modules $V_\l$ for all atypical $\l\in P^+$.

\subsection{Primitive weight graphs}\label{graph}
For a $\gl$-module $V$, a nonzero $\gl_{0}$-highest weight vector
$v\in V$ is called a {\it primitive vector} if $v$ generates an
indecomposable $\fg$-submodule and there exists a $\gl$-submodule
$W$ of $V$ such that $v\notin W$ but $\gl_{+1}v\in W$. If we can
take $W=0$, then $v$ is called a {\it strongly primitive vector} or
a {\it $\gl$-highest weight vector}. The weight of a primitive
vector is called a {\em primitive weight}, and the weight of a
strongly primitive vector a {\em strongly primitive weight} or a
{\em $\gl$-highest weight}.
For a primitive weight $\l$ of $V$, we often use $v_\l$ to denote a
primitive vector of weight $\l$.

\begin{notation}\label{notation-PV}\rm Denote by $P(V)$ the set
of primitive weights with multiplicities. For $\mu,\nu\in P(V)$, if
$\mu\ne\nu$ and $v_\nu\in U(\gl)v_\mu$, we say that $\nu$ is {\em
derived from} $\mu$ and write
$$\mbox{
$\nu\dlar\mu$ \ \ \ or \ \ \ $\mu\drar\nu$. }$$ If $\mu\drar\nu$ and
there exists no $\l\in P(V)$ such that $\mu\drar\l\drar\nu$, then we
say that $\nu$ is {\it directly derived from} $\mu$ and write
$$\mbox{
$\mu\rrar\nu$ \ \ \ or \ \ \ $\nu\llar\mu$. }$$  Sometimes for
convenience, we also use symbols $\mu\link\l$ to denote either
$\mu\rrar\l$ or $\mu\llar\l$. Suppose $V'$ is a submodule of $V$ and
$\mu\in P(V)\bs P(V')$, we use $V'\link\mu$ to indicate
$\nu\link\mu$ for some $\nu\in P(V')$.\end{notation}
\begin{definition}
\label{defi6.1}\rm (cf.~\cite[Definition 6.2]{SuZ1}) Suppose every
composition factor of $V$ is a highest 
weight module. 
Then we can associate
$P(V)$ with a directed graph, still denoted by $P(V)$, in the
following way: the vertices of the graph are elements of $P(V)$. Two
elements $\l$ and $\mu$ are connected by a single directed edge
$($i.e., the two weights are linked$)$ pointing toward $\mu$ if and
only if $\mu$ is directly derived from $\l$. We shall call this
graph the {\it primitive weight graph of $V$}.
\end{definition}

A {\it full subgraph} $S$ of $P(V)$ is a subset of $P(V)$ which
contains all the edges linking vertices of $S$. We call a full
subgraph $S$ {\it closed} if for any $\eta\in P(V)$ and $ \mu,\nu\in
S$,
\equan{closed} { \mu\drar\eta\drar\nu \ \ \ \ \
\Longrightarrow \ \ \ \ \eta\in S. }

It is clear that a module is indecomposable if and only if its
primitive weight graph is {\it connected} (in the usual sense). It
is also clear that a full subgraph of $P(V)$ corresponds to a
subquotient of $V$ if and only if it is closed. Thus a full subgraph
with only $2$ weights is always  closed.


For a directed graph $\G$, we denote by $M(\G)$ any module with
primitive weight graph $\G$ if such a module exists. If $\G$ is a
closed full subgraph of $P(V)$, then $M(\G)$ always exists, which is
a subquotient of $V$.

\subsection{Some technical lemmas}
Let $P^{0+}_{\bar\l}$ be the set of integral $\gl_0$-dominant
atypical weights with atypical type $\bar\l$, and set
$P^+_{\bar\l}=P^+\cap P^{0+}_{\bar\l}$. If $\l\in P^{0+}_{\bar\l}$
has atypical root $\g\in\D_1^+$, we let $a_\pm\in\N=\{1,2,...\}$ be the
smallest such that $\l+a_+\g$ and $\l-a_-\g$ are $\gl_0$-regular
in the sense that $(\l+\rho\pm a_\pm\g, \alpha)\ne 0$
for all roots $\alpha$ of $\gl_0$. If $\g=\d + \es_\ell$ or $\d-\es_\ell$,
then $a_\pm$ are the smallest positive integers such that
$|\tilde\l_\ell\pm a_\pm|\notin S(\bar\l)$
(cf.~(\ref{atypical-type})). Now we defined $\l\nex,\l\pre\in
P^{0+}_{\bar\l}$ by
\equa{l-nex-pre}{\begin{array}{llll}
\l\nex=(\l+a_+\g)^+,&\l\pre=(\l-a_-\g)^+,
\end{array}}
where we recall that given an integral weight $\mu$,
we denote by $\mu^+$ the integral $\gl_0$-dominant weight
$W_0$-conjugate to $\mu$ under the dot action (\ref{dot-action}).
We call $\nex$ and $\pre$ the {\it raising} and
{\it lowering} operators. Note that
\equa{si-nex}{(\l^\si)\nex=(\l\pre)^\si,\ \
(\l^\si)\pre=(\l\nex)^\si\mbox{ \ \ if \ }\tilde\l_0\ne0.}
\begin{definition} \label{l-0=0=l-m=0+}\rm
Keep notation introduced above. If $\tilde\l_0\ne0$, we let
$\l\pre_+=\l\pre$ and $\l\nex_+=\l\nex$. If $\tilde\l_0=0$, we take the atypical root $\g=\d-\es_m$
(cf.~Remark \ref{l-0=0=l-m=0})
and let
$$\l\nex_-=\l\nex,\ \ \ \ \l\nex_+=(\l+a_+\g_+)^+,\ \ \ \
\l\pre_-=(\l-a_-\g_+)^+,\ \ \ \ \l\pre_+=\l\pre,$$
where $\g_+=\d+\es_m$, which is another atypical root.
\end{definition}
\begin{convention}\rm \label{con-nothing}
If an undefined symbol appear in an expression, we regard it as
nothing; for instance, in case $\tilde\l_0=\tilde\l_\ell\ne0$,
symbols $\l\nex_-,\,\l\pre_-$ mean nothing.
\end{convention}

\begin{example}\label{nex-pre}
\rm \begin{enumerate}\item  $\l=(1\,|\,2,1,-1)$ when $k=6$. We shall always use $\line(0,1){5}\line(1,0){10}\line(0,1){5}$ to
indicate a tail atypical root (cf.~Definition \ref{tail}), and use
$\line(0,1){5}\put(0,5){\line(1,0){10}\line(0,-1){5}}\hspace*{10pt}$
to indicate a non-tail atypical root. Then
$$\tilde\l=(-1\UPL{10}{30}\,|\,4,2,
-1\UPR{10}{30}),$$ and $a_+=1,$ $a_-=2$. So
$$\widetilde{\l\nex}=(0\,|\,4,2,0),\,\ \ \ \ \widetilde{\l\pre}=(-3\,|\,4,3,-2),$$ and
$\l\nex=(2\,|\,2,1,0),\,\l\pre=(-1\,|\,2,2,-2).$
 \item  If $k=7$
and $\l=(2\,|\,2,0,0)$, then
\equan{ex1}{\tilde\l=(-\frac12\DNL{11}{30}\,|\,\frac92,\frac32,
\frac12\DNR{11}{30}),} and $a_+=2,\,a_-=1$. Thus
$$\mbox{$\widetilde{\l\nex}=(\frac12\,|\,\frac92,\frac32,\frac12),
\,\ \ \ \ \widetilde{\l\pre}=(-\frac52\,|\,\frac92,\frac52,\frac32),$}$$ and
$\l\nex=(3\,|\,2,0,0),\,\l\pre=(0\,|\,2,1,1)$.
\end{enumerate}\end{example}

\begin{remark}\label{rem-l-m}\rm
In the case $k=2m$, there exists an outer automorphism $\tau$ of
$\gl$ induced by the symmetry automorphism
\equa{tau-auto}{\tau(\es_{m-1}-\es_m)=\es_{m-1}+\es_m,} of the
Dynkin diagram, which changes the sign of the $m$-th coordinate of a
weight. Thus the structure of a generalised Verma module $V_\l$ with
$\l_m<0$ is the same as the structure of $V_\l$ with $\l_m>0$.
Therefore, when considering the generalised Verma module $V_\l$ (or
Kac module $K_\l$, or irreducible module $L_\l$), we can always
suppose
\equa{assume-l-m}{\l_m\ge0.}
However, it should be pointed out that $V_\l$ may contain some
primitive weight $\mu$ with $\mu_m<0$.
\end{remark}

In the following, we always assume that $\bar\l$ is a fixed
atypical type such that its coordinates
$\tilde\l_1,...,\tilde\l_{m-1}$ satisfy
$\tilde\l_1>\cdots>\tilde\l_{m-1}\ge0$ (cf.~(\ref{g0-dominant}),
(\ref{atypical-type}) and (\ref{assume-l-m})).
\begin{definition}\label{defi-aty-type-set}
\rm
Let $j$ be the smallest non-negative integer such that $a:=j+1-s$
$\notin S(\bar\l)$ (cf.~(\ref{atypical-type})). We define $\l^{(0)}$
by
\equa{l-(0)}{\widetilde{\l^{(0)}}=\l^{(0)}+\rho=(-a\,|\,\tilde\l_1,...,
 \tilde\l_{m-1-j},a,\tilde\l_{m+1-j},...,\tilde\l_{m}),}
where we have re-labeled $\tilde\l_{m-j},...,\tilde\l_{m-1}$ by
$\tilde\l_{m+1-j},...,\tilde\l_{m}$ for convenience.
\begin{enumerate}
\item \label{l-(i)} If $k=2m+1$ or $k=2m$ with $0\in S(\bar\l)$, we define
$\l^{(\pm i)}$ (for $i>0$) inductively by
$$\mbox{$\l^{(i)}=(\l^{(i-1)})\nex$, \ \ \
$\l^{(-i)}=(\l^{(1-i)})\pre$,}$$
and set $\l_+^{(\pm i)}=\l^{(\pm i)}$ for all $i>0$.
\item
If $k=2m$ and $0\notin S(\bar\l)$, we have $a=0$. Take the atypical root
$\g=\d-\es_m$ for the corresponding $\l^{(0)}$ defined by (\ref{l-(0)}) as
in Definition \ref{l-0=0=l-m=0+},
and define $\l^{(\pm i)}$ for $i>0$ by
\[\l^{(1)}=(\l^{(0)})\nex_+, \ \ \ \l^{(i)}=(\l^{(i-1)})\nex , \ \ \
\ \l^{(-i)}=(\l^{(1-i)})\pre.
\]
[Recall that $\g_+=\d+\es_m$ is also an atypical root of $\l^{(0)}$.]
Let $\l_+^{(\pm i)}=\l^{(\pm i)}$ for $i>0$, and further define $\l_-^{(\pm i)}$
by $$\l_-^{(1)}=(\l^{(0)})\nex,\ \ \ \l_-^{(i)}=(\l_-^{(i-1)})\nex,
\ \ \ \ \ \l_-^{(-1)}=(\l^{(0)})\pre_-,\ \ \
\l_-^{(-i)}=(\l_-^{(1-i)})\pre.$$
\end{enumerate}
\end{definition}
\begin{example}\label{k=4-l=1+-}\rm
If $k=4,\,\bar\l=(1)$, then $\widetilde{\l^{(0)}}=(0\,|\,1,0)$, and
$$\begin{array}{llll}
\wt{\l^{(1)}}=(2\,|\,2,1),&\wt{\l^{(2)}}=(3\,|\,3,1),&\wt{\l^{(-1)}}
=(-2\,|\,2,1),&\wt{\l^{(-2)}}=(-3\,|\,3,1),\\[4pt]
\wt{\l^{(1)}_-}=(2\,|\,2,-1),&\wt{\l^{(2)}_-}=(3\,|\,3,-1),
&\wt{\l^{(-1)}_-}=(-2\,|\,2,-1),& \l^{(-2)}_-=(-3\,|\,3,-1).
\end{array}$$
\end{example}

From the definition and condition (\ref{d-condition}), we
immediately obtain the following.
\begin{lemma}\label{l-0-unique}
\begin{enumerate}
\item
$\!P^{0+}_{\bar\l}\!=\!\{\l^{(i)}_\pm{\sc\,}|{\sc\,}i\!\in\!\Z\},$\,and
$
P^+_{\bar\l}\!=\!\{\l^{(i)}_\pm{\sc\,}|{\sc\,}i\!\in\!\Z_+\}\,($cf.\,Convention
$
\ref{con-nothing})$.
\item
$\l^{(0)}$ is $\fg$-dominant and is the unique $\fg$-dominant weight
among the $\l^{(i)}_\pm$ which has a tail atypical root.
\end{enumerate}
\end{lemma}
Now we shall investigate the generalised Verma module $V_\l$ for $\l\in
P^{0+}_{\bar\l}$ satisfying (\ref{assume-l-m}). By the PBW
Theorem, $U(\gl_{-1}\oplus\gl_{-2})$ has a basis
\[
B=\Big\{f_{\Theta}\,\Big|\,\Theta\in\Z_+\times\{0,1\}^{2m}\Big\},
\]
where each $f_\Theta$ is an ordered product of the form
$
f_\Theta=f_{2s \d}^\th f_{\d-\es_1}^{\th_1}\cdots
f_{\d-\es_m}^{\th_m}f_{\d+\es_m}^{\bar\th_m}\cdots
f_{\d+\es_1}^{\bar\th_1}
$
for
$\Theta=(\th,\th_1,...,\th_m,\bar\th_m,...,\bar\th_1)$.
Define a total order on $B$ by
\equan{order-B}{f_{\Theta}>f_{\Theta'}\ \ \Longleftrightarrow\ \
|\Theta|>|\Theta'|\mbox{ \ or \ }|\Theta|=|\Theta'| \mbox{ but
}\Theta>\Theta',}
where $|\Theta|=\th+\sum_{i=1}^m(\th_i+\bar\th_i)$
is the {\it level} of $\Theta$, and the order on
$\Z_+\times\{0,1\}^{2m}$ is defined lexicographically.
Recall that a nonzero vector $v\in V_\l$ can be uniquely written as
\equa{write-v}{v= b_1v_1+\cdots+b_tv_t,\ \ b_i\in B,\
b_1>b_2>\cdots, \ 0\ne v_i\in L_\l^{(0)}.} We call $b_1v_1$ the {\it
leading term} (cf.~\cite[\S5]{SHK}). A term $b_iv_i$ is called a
{\it prime term} if $v_i\in\C v_\l$. Note that a vector $v$ may have
zero or more than one prime terms. One immediately has
\begin{lemma}
Let $v=gu$ for some $u\in V_\l$ and $g\in U(\gl^-)$. \begin{itemize}\item[{\rm(i)}] If $u$
has no prime term then $v$ has no prime term. \item[{\rm(ii)}] Let
$v'=gu',\,u'\in V_\l$. If $u,u'$ have the same prime terms then
$v,v' $ have the same prime terms.
\end{itemize}
\end{lemma}
\begin{lemma}\label{gl0-highest}
\begin{itemize}
\item[{\rm(i)}] \label{gl0-highest-1}
Let $v_\mu\in V_\l$ be
a $\gl_0$-highest weight vector of weight $\mu$. Then
\equa{g0-highest}{\l-\mu=2s\th\d+\sum\limits_{i=1}^n(\th_i(\d-\es_i)+\bar\th_i(\d+\es_i)),}
for some $\th\in\Z_+,\,\th_i,\bar\th_i\in\{0,1\}.$ Furthermore the
leading term $b_1v_1$ of $v_\mu$ must be a prime term.
\item[{\rm(ii)}] \label{gl0-highest-2}Suppose $v'_\mu=\sum^{t'}_{i=1}b'_iv'_i$ is
another $\gl_0$-highest weight vector with weight $\mu$. If all
prime terms of $v_\mu$ are the same as those of $v'_\mu$, then
$v_\mu=v'_\mu$.\end{itemize}\end{lemma}
\begin{proof} (i) Let
$v_\mu$ be as in \refequa{write-v}. If $v_1\notin\C v_\l$, then
there exists $e\in \gl_0^+,\,ev_1\ne0.$ We have
\equan{e-v-mu}{
ev_\mu= b_1(ev_1)+[e,b_1]v_1+b_2(ev_2)+[e,b_2]v_2+\cdots.}
Obviously
when writing $[e,b_i]$ in terms of linear combination of $B$, we can
see that $b_1(ev_1)$ is the leading term of $ev_\mu$, i.e.,
$ev_\mu\ne0$, contradicting that $v_\mu$ is a $\gl_0$-highest weight
vector. So, $v_1\in\C v_\l$ and $\l-\mu$ is the weight of $b_1$,
i.e., we have \refequa{g0-highest}.
\par
(ii) Let $v= v_\mu- v'_\mu$. If $v\ne0$ (then it must be a
$\gl_0$-highest weight vector), since its prime terms are all
cancelled, $v$ has no prime term, therefore by (i), it is not
$\gl_0$-highest weight vector, a contradiction.
\end{proof}

Recall from Notation \ref{notation-PV} that $P(V_\l)$ is the set of
primitive weights of $V_\l$. Let $M_\lambda$ be any high weight
$\gl$-module with highest weight $\l$, we also define
\equa{P-0-V-l}{\mbox{$P_0(M_\l)=\{\mu\in P^{0+}_{\bar\l}\,|\,\mu$ is a
$\fg_0$-highest weight in $M_\l\}$.}}
Since $\l$ and any weight in
$P(V_\l)$ have the same central character, they have the same
atypicality type by \refequa{central-atypical}. Thus
$P(V_\l)\subset P_0(V_\l)$.

Let
\begin{eqnarray} \label{a-l-m}\label{b-l-m}
\begin{aligned}
a_{\l,\mu}=[V_\l,L_\mu]&=\mbox{the multiplicity of $L_\mu$ in
$V_\l$},\\
b_{\l,\mu}=[V_\l,L^{(0)}_\mu]&=\mbox{the multiplicity of
$\fg_0$-module $L^{(0)}_\mu$ in $V_\l$}.
\end{aligned}
\end{eqnarray}
In defining $b_{\l,\mu}$ we restrict $V_\l$ to a $\fg_0$-module.
Also, in principle the multiplicity $[V_\l,L_\mu]$ of $L_\mu$ in
$V_\l$ needs to be defined as in \cite[\S 5]{RcW}. However it turns
out that $V_\l$ has composition series of finite length, thus the
complications dealt with in \cite[\S 5]{RcW} do not arise in our
setting. Clearly $ a_{\l,\mu}\le b_{\l,\mu}\mbox{\ for \ }\mu\in
P(V_\l). $

We define a partial order on $\fh^*$ by: $\l>\mu$ if and only if
$\l-\mu$ is a sum of positive roots for $\l,\mu\in\fh^*$.

\begin{lemma}\label{lemma-possible-primitive}
Let $\l\in P^{0+}_{\bar\l}$ with $\l_m\ge0$.
\begin{enumerate}\item\label{P-0-subset}
If $k\ne2m$ or $0\in S(\bar\l)$, then
$P_0(V_\l)\subset\{\l,\,\l\pre,\,
\l^\si,\,(\l^\si)\pre,\,(\l^\si)\nex\} $ and
$b_{\l,\l\pre}\le1.$ 
Furthermore, if $\l$ has a tail atypical root, then
$P_0(V_\l)\subset\{\l,\l\pre\}$. \item If $k=2m$ and $0\notin
S(\bar\l)$, then
\begin{enumerate}\item[\rm (i)]
$P_0(V_{\l^{(0)}})\subset\{\l^{(0)},\l^{(-1)}_\pm\}$ and
$b_{\l^{(0)},\l^{(-1)}_\pm}\le1;$
\item[\rm (ii)] $P_0(V_{\l^{(-i)}})\subset\{\l^{(-i)},\l^{(-i-1)}\}$ and
$b_{\l^{(-i)},\l^{(-i-1)}}\le1$ for $i\ge1;$
\item[\rm (iii)]
$P_0(V_{\l_+^{(1)}})\subset\{\l_+^{(1)},\l^{(0)},\l_+^{(-1)}\}$ and
$b_{\l_+^{(1)},\l^{(0)}}\le1;$
\item[\rm(iv)]
$P_0(V_{\l^{(i)}_+})\subset\{\l^{(i)}_+,\l^{(i-1)}_+,\l^{(1-i)}_+,\l^{(-i)}_+,\l^{(-i-1)}_+\}$
and $b_{\l_+^{(i)},\l_+^{(i-1)}}\le1$ for $i>1.$
\end{enumerate}\end{enumerate}
\end{lemma}
\begin{proof} We shall prove (1) only, as the proof for (2) is similar.
Let $\mu\in P_0(V_\l)$. Then $\mu\le\l$. Denote by $\g_\l$ and $\g_\mu$ the
atypical roots of $\l$ and $\mu$ respectively.

\vskip5pt
 {\it Case 1:
Suppose $\g_\l=\d+\es_\ell,\,\g_\mu=\d+\es_t$ and $\ell\le t$.}

Then we can write
\equan{l-mu}{\begin{array}{cccccccccccccccccccc}
\tilde\l&\!\!\!=\!\!\!&(\L_0\UPL{10}{70}\,|\,\L_1,...,&\!\!\!\!\L_{\ell-1},
&\!\!\!\!\L_\ell\UPR{10}{70},&\!\!\!\!\L_{\ell+1},&\!\!\!\!...,
&\!\!\!\!\L_{t-1},&\!\!\!\!\L_t,&\!\!\!\!\L_{t+1},&\!\!\!\!...,\L_m),\\[12pt]
\tilde\mu&\!\!\!=\!\!\!&(M_0\UPL{10}{130}\,|\,\L_1,...,&\!\!\!\!\L_{\ell-1},
&\!\!\!\!\L_{\ell+1},&\!\!\!\!\L_{\ell+2},&\!\!\!\!...,
&\!\!\!\!\L_{t},&\!\!\!\!M_t\UPR{10}{130},&\!\!\!\!\L_{t+1},&\!\!\!\!...,\L_m),
\end{array}}
where $\L_0=\L_\ell,\,M_0=M_t$. By \refequa{g0-highest},
there exists $\Theta\in\Z_+\times\{0,1\}^{2m}$
such that
\equa{mu-in-P(l)}{\begin{array}{llllll} \L_0=M_0+2s
\th+\sum\limits_{i=1}^m(\th_i+\bar\th_i),&
\th_i=\bar\th_i\ (i<\ell\mbox{ or }i>t),\\[12pt]
\L_i=\L_{i+1}-\th_i+\bar\th_i\ (\ell\le i<t), &
\L_t=M_t-\th_t+\bar\th_t.
\end{array} }
Note that $\L_i\ge\L_{i+1}+1$ for all $i$ and $\L_t\ge M_t+1\ge
\L_{t+1}+2$ (the last inequality occurs only when $t<m$) , and
$\th_i,\bar\th_i\in\{0,1\}$. Thus $\th_i=0,\,\bar\th_i=1$ for
$\ell\le i\le t$ and $\th=\th_j=\bar\th_j=0$ for $j<\ell$ or $j>t$.
So
\equa{L-ell-M-t}{\L_\ell=\L_{\ell+1}+1=\L_{\ell+2}+2=...=\L_t+t-\ell=M_t+t+1-\ell.}
(If $\ell=t$, there is no last equality and $\bar\th_\ell$ can be
$0$ or $1$). Thus $\mu=\l$ or $\l\pre$. Since in each case, the
solution $\Theta$ is unique, by Lemma \ref{gl0-highest}(2),
$a_{\l,\mu}\le1$. \vskip5pt

 {\it Case 2:
Suppose $\g_\l=\d+\es_\ell,\,\g_\mu=\d+\es_t$ and $\ell>t$.}

Then $M_0=M_t>\L_{t+1}\ge\L_\ell=\L_0$, which contradicts the first
equation in \refequa{mu-in-P(l)}. Thus $\ell>t$ cannot
occur.\vskip5pt

{\it Case 3: Suppose $\g_\l=\d+\es_\ell,\,\g_\mu=\d-\es_t$ and
$\ell\le t$.}

Then $M_0=-M_t$. We still have \refequa{L-ell-M-t}, thus $\mu$ has
to be $\l^\si$ or $(\l^\si)\nex$ (in this case, the solution of
$\Theta$ might not be unique). \vskip5pt

{\it Case 4: Suppose $\g_\l=\d+\es_\ell,\,\g_\mu=\d-\es_t$ and
$\ell>t$.}

Then again $M_0=-M_t$. We re-write $\tilde\l,\tilde\mu$ as
\equan{l-mu-rewrite}{\begin{array}{cccccccccccccccccccc}
\tilde\l&\!\!\!=\!\!\!&(\L_0\UPL{10}{120}\,|\,\L_1,...,&\!\!\!\!\L_{t-1},
&\!\!\!\!\L_t,&\!\!\!\!\L_{t+1},&\!\!\!\!...,
&\!\!\!\!\L_{\ell-1},&\!\!\!\!\L_\ell\UPR{10}{120},&\!\!\!\!\L_{\ell+1},
&\!\!\!\!...,\L_m),\\[9pt]
\tilde\mu&\!\!\!=\!\!\!&(M_0\DNL{8}{70}\,|\,\L_1,...,&\!\!\!\!\L_{t-1},
&\!\!\!\!M_t\DNR{8}{70},&\!\!\!\!\L_t,&\!\!\!\!...,
&\!\!\!\!\L_{\ell-2},&\!\!\!\!\L_{\ell-1},&\!\!\!\!\L_{\ell+1},&\!\!\!\!...,\L_m).
\end{array}}
Thus
\equan{L-ell-M-t-re} { M_t=\L_t+1=\L_{t+2}+2=...=\L_\ell+\ell-t+1.}
In this case, $\mu$ has to be $(\l^\si)\pre$ (the solution of
$\Theta$ might not be unique).

\vskip5pt

{\it Case 5: Suppose $\g_\l=\d-\es_\ell,\,\g_\mu=\d\pm\es_t$.}

Then $M_0\le\L_0\le0$. As in Case 2, we have $t\le\ell$. Similar to Case 1,
we obtain $\mu=\l$ or $\l\pre$. The atypical root of $\mu$ is
$\g_\mu=\d-\es_t$ in both situations, and the solution for $\Theta$ is unique. Thus
$a_{\l,\mu}\le1$. This prove the lemma.
\end{proof}

For $\l,\mu\in\fh^*$ such that $\l-\mu$ is a sum of distinct
positive odd roots, i.e., $\l-\mu=\sum_{\b\in\Si}\b$ for some
$\Si\in\D_{\bar1}^+$, we introduce the {\it relative level}
$|\l-\mu|$ to be the cardinality $|\Si|$, which is also equal to
$\l_0-\mu_0$. For any integral (not necessarily $\fg_0$-dominant)
weight $\l$ with the atypical root $\g_\l$ (cf.~Remark
\ref{l-0=0=l-m=0}), we introduce the Bernstein-Leites formula \cite{BL}
\equa{B-L-formula}{\chi^{\rm
BL}_\l={\dis\frac1{R_{\bar0}}}\sum\limits_{w\in
W}\sign(w)\Big(e^{\l+\rho_{\bar0}}
\prod\limits_{\b\in\D_1^+\bs\{\g_\l\}}(1+e^{-\b})\Big).
}
For convenience, we introduce the following  notation:
\equa{chi-def} {
\begin{aligned}   &\text{$\chi^{\rm BL}_{0,\l}$  denotes  the right-hand side of
(\ref{B-L-formula}) with $W$ replaced by $W_0$},\\
&\text{$\chi^V_\l$ denotes the right-hand side of (\ref{typical-char}) for
$\l\in P$}.
\end{aligned}
}
Then by recalling the definition of regular weights
(immediately before (\ref{g0-dominant})),
one easily sees that \equa{ch-v-l-not=0}{\mbox{ $\chi^V_\l\ne0$ \ \
$\Lra$ \ \  $\l$ is regular.}}

\begin{lemma} \label{lemma-tail-primitive}
Assume that $\l\in P^+_{\bar\l}$ with $\l_m\ge0$ has a tail atypical root
$\g_\l=\d-\es_\ell$.
\begin{enumerate}\item
If $\l\ne\l^{(0)}$ or $0\in S(\bar\l)$, then
$P(V_\l)=\{\l,\l\pre\}$, and the primitive weight graph of $V_\l$ is
$\l\rrar\l\pre$. Furthermore,
\[
\ch L_\l=\chi^{\rm BL}_{0,\l}, \quad \ch
L_{\l^{(0)}}=\frac12\chi^{\rm BL}_{\l^{(0)}}.
\]
\item
 If $\l=\l^{(0)}$ and $0\notin S(\bar\l)$, then
 $P(V_{\l^{(0)}})=\{\l^{(0)},\l^{(-1)}_\pm\}$ with the primitive weight
 graph
\equa{tail-graph-0} {\mbox{ }\ \ \ \l^{(0)}\
\put(0,5){$\nearrow\put(0,5){\footnotesize$\l^{(-1)}_+$}$}
\put(0,-5){$\searrow\put(0,-5){\footnotesize$\l^{(-1)}_-$}$}
\hspace*{9ex}.} Furthermore,
\[
\ch L_{\l^{(0)}}=\chi^{\rm BL}_{\l^{(0)}}.
\]
\end{enumerate}
\end{lemma}
\begin{proof} (1) Since $\l$ has a tail atypical root,
we have $\l=\l^{(i)}$ for some $i\le0$. As $L_{\l^{(0)}}$ is finite
dimensional, we have $L_{\l^{(0)}}\ne V_{\l^{(0)}}$, i.e., $\mu\in
P(V_{\l^{(0)}})$ for some $\mu\ne\l^{(0)}$. By Lemma
\ref{lemma-possible-primitive}, $\mu=\l\pre=\l^{(-1)}$. Now one can
use the arguments in the paragraph after (\ref{ch-l-1}) to prove
$\l^{(i-1)}\in P(V_{\l^{(i)}})$ for $i<0$. For the purpose of the
next lemma and other purposes, we provides below a way to explicitly
construct the primitive vector $v_\mu$ (cf.~\refequa{v-mu} and
\refequa{v-mu-i}) for $\mu=\l\pre$.

 By Case 5 in the proof of
Lemma \ref{lemma-possible-primitive}, we have
\equan{l-mu-tail}{\begin{array}{cccccccccccccccccccc}
\tilde\l&\!\!\!\!=\!\!\!\!&(\L_0\DNL{8}{150}\,|\,\L_1,...,&\!\!\!\!\L_{t-1},
&\!\!\!\!\L_\ell\!+\!\ell\!-\!t,
&\!\!\!\!\L_\ell\!+\!\ell\!-\!t\!-\!1,&\!\!\!\!...,
&\!\!\!\!\L_{\ell}\!+\!1,&\!\!\!\!\!\!\!\!\!\!\!\!\L_\ell\DNR{8}{120},\!\!\!\!\!\!\!\!
&\!\!\!\!\L_{\ell+1},&\!\!\!\!...,\L_m),\\[15pt]
\tilde\mu&\!\!\!\!=\!\!\!\!&(M_0\DNL{8}{70}\,|\,\L_1,...,&\!\!\!\!\L_{t-1},
&\!\!\!\!\L_\ell\!+\!\ell\!-\!t\!+\!1,
\!\!\!\!\!\!\!\!\!\!{}\DNR{8}{70}\ \ \
&\!\!\!\!\L_\ell\!+\!\ell\!-\!t,&\!\!\!\!...,
&\!\!\!\!\L_{\ell}\!+\!2,&\!\!\!\!\L_{\ell}\!+\!1,&\!\!\!\!\L_{\ell+1},&\!\!\!\!...,\L_m),
\end{array}}
where $t\le\ell$. We define $\mu^{(0)}=\l$ and
$\mu^{(i)}=\mu^{(i-1)}-(\d-\es_{\ell+1-i})$. Then
$\mu=\mu^{(\ell-t+1)}$.\vskip5pt

{\it Case 1: $t=\ell$.}

We want to construct, from the highest weight vector $v_\l$ of $V_\l$, a
$\gl_0$-highest weight vector $v_\mu$ with weight $\mu$. By Lemma
\ref{gl0-highest}(1), $v_\mu$ should contain a leading term which
has to be $c f_{\d-\es_\ell}v_\l$ for some nonzero $c\in\C$. Set
$I=\{1,...,\ell-1\}$, and  for $S=\{i_1<\cdots<i_p\}\subset I$, we
define $f_S$ by
\equa{define-fS}{f_S=\left\{\begin{array}{cl}f_{\d-\es_{i_1}}f_{\es_{i_1}-\es_{i_2}}\cdots
f_{\es_{i_{p-1}}-\es_{i_p}}f_{\es_{i_p}-\es_\ell}&\mbox{if \
}S\ne\emptyset,\\[4pt]
f_{\d-\es_\ell}&\mbox{if \ }S=\emptyset.\end{array}\right.}
In order for $v_\mu$ to have weight $\mu$, we can suppose that
$v_\mu$ has the form
\equa{v-mu}{v_\mu=\sum\limits_{S\subset I}c_S f_S v_\l,}
for some $c_S\in\C$. Now if we define $c_S$ uniquely by
\equa{define-cS}{
c_S=\left\{\begin{array}{cl}1&\mbox{if }S=I,\\[4pt]
-(\l_q\!-\!\l_\ell\!+\!\ell\!-\!q\!-\!1)c_{S'}\!\!\!&\mbox{if
}S\!\subsetneqq\! I,\,S'\!=\!S\!\cup\!\{q\},\,q\!=\!{\rm max}(I\bs
S),\end{array}\right. }
then we can prove
\begin{claim}\label{claim-lemma1}
 $v_\mu$ is a $\gl_0$-highest weight vector.
\end{claim}

First note that if $t=1$ then \refequa{v-mu}
is simply reduced to $v_\mu=f_{\d-\es_1}v_\l$. In general, note from
$t=\ell$ that $\L_{\ell-1}>\L_\ell+1$, i.e., $\l_{\ell-1}>\l_\ell$,
and that $q\le\ell-1$, we have
$$\l_q-\l_\ell+\ell-q-1\ge\l_{\ell-1}-\l_\ell+\ell-q-1>\ell-q-1\ge0,$$
and so $c_S\ne0$ for all $S\subset I$. In particular, the leading
term of $v_\mu$ is $c_\emptyset f_{\d-\es_\ell}v_\l\ne0$, i.e.,
$v_\mu\ne0$.

To prove the claim, we only need to prove $ev_\mu=0$ for
\equa{e-in}{\mbox{$e\in\{e_{\es_a-\es_{a+1}}\ (1\le a<m),\
e_{\es_{m-1}+\es_m}$ (if $k=2m$), $e_{\es_m}$ (if $k=2m+1$)\,$\}.$}}
Note from Notation \ref{e-f-a} that (in the following we denote
$\es_0=\d$ so that $b$ can be $0$)
\equa{e-a-j,f-a-j}{[e_{\es_a-\es_{a+1}},f_{\es_b-\es_c}]=\left\{\begin{array}{ll}
h_{\es_a-\es_{a+1}}&\mbox{if }b=a,\,\ c=a+1,\\[4pt]
-f_{\es_{a+1}-\es_c}&\mbox{if }b=a,\,\ c>a+1,\\[4pt]
f_{\es_b-\es_a}&\mbox{if }b<a,\,\ c=a+1,\\[4pt]
0&\mbox{otherwise},
\end{array}\right.}
where $b<c$. We see that $e_{\es_{a}-\es_{a+1}}\,(\ell\le a\le m-1),
\,e_{\es_{m-1}+\es_m}$ (if $k=2m$), $e_{\es_m}$ (if $k=2m+1$)
commute with $f_S$ for all $S\subset I$. Thus it suffices to
consider $e=e_{\es_{a}-\es_{a+1}}$ for $1\le a<\ell$.

For any $S_0=\{i_1<\cdots<i_x <j_1<\cdots<j_p\}\subset I$ with
$i_x<a,\ a+1<j_1$ (where $x$ or $p$ can be zero). Let
$S_1=S_0\cup\{a,a+1\},$ $S_2=S_0\cup\{a\},$ $S_3=S_0\cup\{a+1\}$. We
want to prove
\equa{S0-S1-S2-S3}{c_{S_2}=-(\l_a-\l_\ell+\ell-a-1)c_{S_1},\ \
c_{S_3}=-(\l_{a+1}-\l_\ell+\ell-a-2)c_{S_1}.} Set $b={\rm max}(I\bs
S_0)$. If $b=a+1$, then we have \refequa{S0-S1-S2-S3} by definition
\refequa{define-cS}. If $b>a+1$, then \refequa{define-cS} shows
$c_{S_i}=-(\l_b-\l_\ell+\ell-b-1)c_{S'_i}$ for $S'_i=S_i\cup\{b\}$
and $i=1,2,3$. Thus \refequa{S0-S1-S2-S3} can be proved by induction
on ${\rm max}(I\bs S)$. From \refequa{e-a-j,f-a-j} and
\refequa{S0-S1-S2-S3}, we obtain \equa{lemma-e-f}{
[e_{\es_a-\es_{a+1}},\sum\limits_{i=1}^3
c_{S_i}f_{S_i}]v_\l=((\l_a-\l_{a+1}+1)c_{S_1}+c_{S_2}-c_{S_3})\hat
f_{S_1}=0,} where $\hat f_{S_1}$ is defined as $f_{S_1}$ in
\refequa{define-fS} but with the factor $f_{\es_a-\es_{a+1}}$
removed. Note that the family $\{S\,|\,S\subset I\}$ of subsets of
$I$ can be divided into a disjoint union of blocks, each block has
the form $\{S_0,S_1,S_2,S_3\}$ defined as above. Thus
\refequa{lemma-e-f} together with $[e_{\es_a-\es_{a+1}},f_{S_0}]=0$
implies $e_{\es_a-\es_{a+1}}v_\mu=0$, and so the claim is proved.
\begin{claim}
$e_{\d-\es_1}v_\mu=0.$
\end{claim}

For any $S=\{i_1<\cdots<i_p\}\subset I$ with $1<i_1$, let
$S_1=S\cup\{1\}$. Then as in the proof of \refequa{S0-S1-S2-S3}, we
have $c_S=-(\l_1-\l_\ell+\ell-2)c_{S_1}$, and
$$[e_{\d-\es_1},c_Sf_S+c_{S_1}f_{S_1}]=(-(\l_1-\l_\ell+\ell-2)+h_{\d-\es_1})\hat
f_{S_1}v_\l=0,$$ where $\hat f_{S_1}$ is defined as $f_{S_1}$ in
\refequa{define-fS} but with the factor $f_{\d-\es_1}$ removed, and
the last equality follows from the fact that $f_{S_1}v_\l$ is an
eigenvector of $h_{\d-\es_1}$ with eigenvalue
$$\l_0-\l_1-1=\L_0+m-s-\l_1-1=-\L_\ell+m-s-\l_1-1=\l_1-\l_\ell+\ell-2.$$
This proves the claim. Thus the lemma is proved in this
case.\vskip5pt

{\it Case 2: $t<\ell$.}

We inductively construct $v_{\mu^{(i)}}$ for $i=1,...,\ell-t+1$ as
follows: \begin{eqnarray}\label{v-mu-i}
v_{\mu^{(i)}}\!\!\!&=&\!\!\!\mbox{$\sum\limits_{S^{(i)}\subset
I^{(i)}}$}c^{(i)}_{S^{(i)}} f^{(i)}_{S^{(i)}}v_{\mu^{(i-1)}},\mbox{
\
 where}\\
\nonumber
I^{(i)}\!\!\!&=&\!\!\!\biggl\{\begin{array}{lll}\{1,...,\ell-1-i\}&\mbox{if
\ }
i\le\ell-t,\\[4pt]
\{1,...,t-1\}&\mbox{if \ }i=\ell+1-t,\end{array}
\end{eqnarray}and $f_{S^{(i)}}^{(i)},
c_{S^{(i)}}^{(i)}$ are defined as in \refequa{define-fS},
\refequa{define-cS} with $\ell$ replaced by $\ell+1-i$ and $\l$
replaced by $\mu^{(i-1)}$.

First note from the definition of $c_{S^{(i)}}^{(i)}$ in
\refequa{define-cS} that when $i\le \ell-t$, we have
$$\mu^{(i-1)}_q-\mu^{(i-1)}_{\ell+1-i}+(\ell+1-i)-q-1\ge \ell-i-q>0\mbox{ \ since }q\in I^{(i)}, $$
and when $i=\ell-t+1$, we have
$$\mu^{(i-1)}_q-\mu^{(i-1)}_t+t-q-1\ge\mu^{(i-1)}_q-\mu^{(i-1)}_t=\l_q-\l_t>
0.$$ Thus $c_{S^{(i)}}^{(i)}\ne0$ for all $S^{(i)}\subset I^{(i)}$.
Also note that when we write $v_{\mu^{(i)}}$ in terms of
\refequa{write-v}, it produces only one possible leading term, and
the coefficient of the leading term is
$\prod_{i=1}^{\ell-t+1}c^{(i)}_\emptyset\ne0$. In particular
$v_{\mu^{(i)}}\ne0$.

The same arguments in Case 1 show that $e_\a v_{\mu^{(1)}}=0$ for
all simple roots $\a\in\Pi$ except $\a=\es_{\ell-1}-\es_\ell$.
Induction on $i$ shows \begin{eqnarray}\label{e-a-mu-i} e_\a
v_{\mu^{(i)}}\!\!\!&=&\!\!\!0\mbox{ \ \ for all \
}\a\in\Pi\bs\Pi^{(i)},\mbox{ \ where}\\
\nonumber
 \Pi^{(i)}&=&\!\!\!\biggl\{\begin{array}{ll}
\{\es_j-\es_{j+1}\,|\,\ell-i\le
j\le\ell-1\}&\mbox{if \ }i\le \ell-t,\\[4pt]
\{\es_j-\es_{j+1}\,|\,t\le j\le\ell-1\}&\mbox{if \
}i=\ell+1-t.\end{array}\end{eqnarray}
\begin{claim}
$v_\mu=v_{\mu^{(\ell+1-t)}}$ is a primitive weight vector.
\end{claim}

Otherwise there exists some element $x\in
U(\gl^+)=U(\gl_{\bar0}^+)U(\gl_{+1})$ of the weight, say, $\a$ (here
$U(\gl_{+1})$ denotes the skew-symmetry tensor space of $\gl_{+1}$),
such that $v=xv_\mu\ne0$ is primitive. Since there is no other
primitive weight between $\l$ and $\mu=\l\pre$ by Lemma
\ref{lemma-possible-primitive}, it follows that $v$ must have the
weight $\l$, i.e., $\l=\a+\mu$. Observe from the construction of
$v_\mu$ that for all $\b\in\DOP$ with $\b>\d-\es_\ell$, we have
$e_\b v_\mu=0$. Since the relative level of $\l$ and $\mu$ is
$|\l-\mu|=\ell+1-t$, there must be $\ell+1-t$ positive odd roots in
order to produce $\a=\l-\mu=\sum_{i=t}^\ell(\d-\es_i)$, thus for any
term in $x$, there must be a root vector $e_{\d-\es_j}$ occurring as
a factor for some $j\le t$. But $e_{\d-\es_j}v_\mu=0$ by
\refequa{e-a-mu-i}. Thus $v=xv_\mu=0$, a contradiction. This proves
the first statement of Lemma \ref{lemma-tail-primitive}(1).

Thus the primitive weight graph of $V_{\l^{(i)}}$ for $i\le0$ is given by
$\l^{(i)}\rrar\l^{(i-1)}$. Hence, $\ch
V_{\l^{(i)}}=\ch L_{\l^{(i)}}+\ch L_{\l^{(i-1)}}$ and it follows that
\[ \ch L_{\l^{(i)}}=\sum_{j=0}^\infty(-1)^j\ch V_{\l^{(i-k)}}=\chi^{\rm
BL}_{0,\l^{(i)}},
\]
where the last equality can be obtained from
(\ref{l-nex-pre}) and (\ref{ch-v-l-not=0}), or from the same
arguments in \cite{VHKT}. To prove $\ch
L_{\l^{(0)}}=\frac12\chi^{\rm BL}_{\l^{(0)}}$, suppose the atypical
root of $\l^{(0)}$ is $\g=\d-\es_\ell$. Note from (\ref{l=si})  that
$(\l^{(0)})^\si=\l^{(1)}=\l^{(0)}+2a\d$, where $a$ is defined in
(\ref{l-(0)}), so
$\si(\l^{(0)}+\rho_{\bar0})=\si(\l^{(0)}+\rho+\rho_1)=
\l^{(1)}+\rho_{\bar0}-2\rho_1$, and
\begin{eqnarray}\label{D-1===}
\si\Big(\mbox{$\prod\limits_{\b\in\D^+_1\bs\{\d-\es_\ell\}}$}
(1+e^{-\b})\Big)
&=&e^{2\rho_1-\d-\es_\ell}\mbox{$\prod\limits_{\b\in\D^+_1\bs\{\d+\es_\ell\}}$}(1+e^{-\b})
\nonumber\\
&=&D_1\mbox{$\sum\limits_{i=1}^\infty$}(-1)^{i-1}e^{2\rho_1-i(\d+\es_\ell)},
\end{eqnarray}
where $D_1=\prod_{\b\in\D_1}(1+e^{-\b})$  is $W_0$-invariant. Thus
\equa{SOSOS}{\si(\chi^{\rm
BL}_{0,\l^{(0)}})=\sum_{i=1}^\infty(-1)^{i-1}\chi^V_{\l^{(1)}-i(\d+\es_\ell)}$
$=(-1)^{2a-1}\chi^{\rm BL}_{0,\l^{(1)}-2a(\d+\es_\ell)}=-\chi^{\rm
BL}_{\l^{(0)}},} where the last equality is proved as follows.

If $s=1$, i.e., $k=2m$, then $a\in\Z$ and $0\in S(\bar\l)$, so
$(-1)^{2a-1}=-1$, and
$\l^{(0)}=\theta\cdot(\l^{(1)}-2a(\d+\es_\ell))$, where the action
is the dot action defined by (\ref{dot-action}), and $\theta\in W_0$
with $\sign(\theta)=1$ is the unique element changing the signs of
$\ell$-th and $m$-th coordinates. If $s=\frac12$, i.e., $k=2m+1$,
then $2a-1\in2\Z$ and in this case $\theta$ with $\sign(\theta)=-1$
is the unique element changing the sign of $\ell$-th. In any case,
we have (\ref{SOSOS}), which together with $\ch
L_{\l^{(0)}}=\chi^{\rm BL}_{0,\l^{(0)}}$ implies   $\ch
L_{\l^{(0)}}=\frac12\chi^{\rm BL}_{\l^{(0)}}$.

(2) We can prove (\ref{tail-graph-0}) similarly as in part (1) of the proof.
From this we obtain
\[
\ch L_{\l^{(0)}}=\chi^V_{\l^{(0)}}-\chi^{\rm
BL}_{0,\l_+^{(-1)}}-\chi^{\rm BL}_{0,\l_-^{(-1)}}=\chi^{\rm
BL}_{0,\l^{(0)}}-\chi^{\rm BL}_{0,\l_-^{(-1)}}=\chi^{\rm
BL}_{\l^{(0)}}.
\]
This completes the proof.
\end{proof}
\begin{lemma}\label{lemma-two-nontail} Let $\l\in P^+_{\bar\l}$ with $\l_m\ge0$. Assume
that $\mu=\l\pre\in P^+_{\bar\l}$ and both $\l$ and $\mu$ have the same non-tail atypical
root $\g=\d+\es_1$. Then $\mu\in P(V_\l)$ and
$a_{\l,\mu}=b_{\l,\mu}=1$.
\end{lemma}
\begin{proof}In principle we can prove the lemma by formal arguments, but we prefer to
construct the corresponding highest weight vectors explicitly, as
they give extra information. Now we have
$\tilde\l_0=\tilde\l_1>\tilde\l_2+1$ and
$\tilde\mu_0=\tilde\mu_1=\tilde\l_0-1$. From this and the proof of
Lemma \ref{lemma-tail-primitive}, we immediately obtain
$b_{\l,\mu}\le1$. We assume $k=2m+1$ as the case $k=2m$ is similar.
We can decompose $\g$ as
$$\g=\d+\es_1=\tau_{-m-1}+\tau_{-m}+\cdots+\tau_{-2}+\tau_{-1}+\tau_{1}+\tau_{2}+\cdots+\tau_m,$$
where
\begin{eqnarray*}\label{alpha-1}
\!\!\!\!\!\!\!\!\!\!\!\!&\!\!\!\!\!\!\!\!\!\!\!\!&
\tau_{-m-1}=\d-\es_1,\,\ \ \ \ \tau_{i}=\es_{m+i+1}-\es_{m+i+2}\,\
(-m\le
i\le -2),\\
\label{alpha-2} \!\!\!\!\!\!\!\!\!\!\!\!&\!\!\!\!\!\!\!\!\!\!\!\!&
\tau_{-1}=\tau_1=\es_m,\,\ \ \ \ \,\tau_i=\es_{m-i+1}-\es_{m-i+2}\,\
(2\le i\le m).
\end{eqnarray*}
For any $i,j$ with $-m-1\le i\le j\le m$, we denote
\equan{alpha-ij}{\tau_{ij}=\sum\limits_{i\le p\le j,\,p\ne0}\tau_p.}
Then $\tau_{ij}$ is not a root if and only if $j=-i$, in this case
we set $f_{\tau_{ij}}=0$. Set $I=\{-m,-m+1,...,-1,1,2,...,m\}$. For
any subset $S=\{i_1<i_2<\cdots<i_p\}$ of $I$ satisfying
(\ref{S-cond}), we define
\equan{new-f-S-B}{
f_S=f_{\tau_{-m-1,i_1-1}}f_{\tau_{i_1,i_2-1}}\cdots
f_{\tau_{i_{p-1},i_p-1}}f_{\tau_{i_p,m}}.}
Note that if
$S=\emptyset$, then $f_S=f_{\tau_{-m-1,m}}=f_{\d+\es_1}$. Also,
$f_S=0$ if $i_{p'+1}=-i_{p'}+1$ for some $p'<p$ or $i_{p}=-m$. So we
suppose
\equa{S-cond}{i_{p'+1}\ne-i_{p'}+1\ \,(1\le p'< p),\ \ \ \
i_{p}\ne-m.}
Now we define $v_\mu=\sum_{S\subset I}c_Sf_Sv_\l$ such
that $c_S\in\C$ is defined by
\equan{c-S-1}{
c_S=\left\{\begin{array}{cl} 1&\mbox{if \ }S=I,\\[4pt]
(\l_1-\l_{m+2-q}+m-q)c_{S'}&\mbox{if \ }S\subsetneqq
I,\,1<q\le m,\\[4pt]
(\l_m+1)c_{S'}&\mbox{if \ }S\subsetneqq
I,\,q=1,\\[4pt]
-(\l_q-\l_1-q)c_{S'}&\mbox{if \ }S\subsetneqq I,\,q<0,
\end{array}\right.}
where  $q$ is the maximal integer in $I\bs S$ such that
$S'=S\cup\{q\}$ satisfies condition (\ref{S-cond}). Note that in all
cases $c_S\ne0$ as long as $c_{S'}\ne0$. Now as in Case 1 in the
proof of Lemma \ref{lemma-tail-primitive}, we see $v_\mu$ is a
$\fg$-highest weight vector.
\end{proof}

\subsection{Main results on generalised Verma modules}
Now we can prove the following result on the structure of
generalised Verma modules. Fix an atypicality type ${\bar\l}$, and
let $\l^{(i)} \in P^{0+}_{\bar\l}$ be the weights defined by
Definition \ref{defi-aty-type-set}.

\begin{remark}\rm
Recall from Lemma \ref{l-0-unique} that
$P^{0+}_{\bar\l}=\{\l^{(i)}_\pm \mid {\sc\,}i\in\Z\}$ with
$\l^{(i)}_+=\l^{(i)}$ for all $i$. By Remark \ref{rem-l-m}, the
generalised Verma modules $V_{\l^{(i)}_+}$ and $V_{\l^{(i)}_-}$ have
the same structure when $\l^{(i)}_-$ is defined.
\end{remark}

\begin{theorem}\label{theorem-verma-module}
For $\l=\l^{(i)}\in P^{0+}_{\bar\l}$, the composition factors of the
generalised Verma module $V_\l$ all have multiplicity $1$.
Furthermore, the primitive weight graph of $V_\l$ can be described
as follows.
\begin{enumerate}
\item
If $k=2m+1$ or $0\in S(\bar\l)$ $($in this case,
$(\l^{(i)})^\si=\l^{(1-i)})$, then \equa{Stru-Kac-mod}{
\begin{array}{c}\!\!\!\!\!\!\!\!\l=\l^{(i)}\\[-4pt]\downarrow\\\l^{(i-1)}\\[2pt]\sc i\le0\end{array}\
, \ \ \ \ \ \
\begin{array}{c}\!\!\!\!\!\!\!\!\l=\l^{(1)}\\[-4pt]\downarrow\\\l^{(-1)}\\ \end{array}\ ,\
\ \ \ \
\raisebox{15pt}{\mbox{$\begin{array}{c}\!\!\!\!\!\!\!\!\!\!\!\!\l=\l^{(2)}\\
\downarrow\searrow\\[-2pt]\ \ \ \l^{(0)} \ \l^{(1)}\end{array}
\put(-17,-15){$\vector(-1,-2){8}$}\put(-32,-42){$\l^{(-1)}$}\put(-40,-15){$\vector(1,-2){8}$}
\put(-42,10){$\vector(-1,-3){20}$}\put(-68,-60){$\l^{(-2)}$}\put(-50,-47){$\vector(2,1){16}$}$}}\
,\ \ \ \ \ \ \raisebox{25pt}{\mbox{$\begin{array}{c}\!\!\!\!\!\!\!\!\l=\l^{(i)}\\
\phantom{\downarrow}\searrow\\[-2pt]\ \ \ \phantom{\l^{(0)}} \ \l^{(i-1)}\end{array}
\put(-17,-15){$\vector(-1,-2){8}$}
\put(-32,-42){$\l^{(1-i)}$}
\put(-42,10){$\vector(-1,-3){20}$}\put(-68,-60){$\l^{(-i)}$}\put(-50,-47){$\vector(2,1){16}$}\put(-45,-68){$\sc
i\ge3$}$}}\ .}
\item
If $k=2m$ and $0\notin S(\bar\l)$ $($in this case,
$(\l^{(i)})^\si=\l^{(-i)})$, then \equa{Stru-Kac-mod-1}{
\begin{array}{c}\!\!\!\!\!\!\!\!\l=\l^{(i)}\\[-4pt]\downarrow\\\l^{(i-1)}\\[2pt]\sc i<0\end{array}\
, \ \ \ \ \ \
\begin{array}{c}\!\!\!\!\!\!\!\!\l=\l^{(0)}\\[-4pt]\swarrow\searrow\\\l^{(-1)}_+\ \ \l^{(-1)}_-\\ \end{array}\ ,
\ \ \ \ \ \ \raisebox{25pt}{\mbox{$\begin{array}{c}\!\!\!\!\!\!\!\!\l=\l^{(i)}\\
\phantom{\downarrow}\searrow\\[-2pt]\ \ \ \phantom{\l^{(0)}} \ \l^{(i-1)}\end{array}
\put(-17,-15){$\vector(-1,-2){8}$}
\put(-32,-42){$\l^{(-i)}$}
\put(-42,10){$\vector(-1,-3){20}$}\put(-68,-60){$\l^{(-i-1)}$}\put(-50,-47){$\vector(2,1){16}$}\put(-45,-68){$\sc
i\ge1$}$}}\ .}
\end{enumerate}
\end{theorem}
\begin{proof} The fact that the multiplicity of each composition
factor is $1$ has  already been proven in Lemma
\ref{lemma-possible-primitive} for some cases, and the remaining
cases are treated presently. As for the rest of the theorem, i.e.,
statements (1) and (2), we note that their proofs are similar, thus
we consider (1) only.

First we need the following: for any integral $\fg_0$-dominant
weight $\l,\mu$ with $\mu\le\l$, using the well-known Weyl character
formula of a $\fg_0$-module, by (\ref{b-l-m}) and
(\ref{char-Verma-for}), we have
\equa{b-l-mu=}{b_{\l,\mu}=\sum\limits_{\stackrel{\stackrel{\ssc
p\in\Z_+,\,S\subset\D_1^+,\,|S|=|\l-\mu|-2p}{\ssc
\nu:=\l-\sum_{\b\in S}\b-2p\d\ {\rm is\ regular}}}{\ssc w\in
W_0,\,\nu=w\cdot\mu}}\sign(w).} We already have the first graph of
(\ref{Stru-Kac-mod}) by Lemma \ref{lemma-tail-primitive}. To prove
the second, we first use (\ref{b-l-mu=}) to prove
\equa{b-l-1-l-0=0}{b_{\l^{(1)},\l^{(0)}}=0.} For convenience,
suppose $k=2m+1$ and $\wt{\l^{(0)}}$ has the form (\ref{l-(0)}).
Then $a=j+\frac12$, $\tilde\l_{m-i}=i+\frac12$ for $j-1\ge i\ge 0$,
and $\wt{\l^{(1)}}$ has the form (\ref{l-(0)}) with $-a$ replaced by
$a$ in the first place. Suppose (\ref{b-l-mu=}) with
$\l=\l^{(1)},\,\mu=\l^{(0)}$ has a term $\sign({w_0})$ with
$p_0\in\Z_+,\,S_0\subset\D_1^+,\,w_0\in W_0$ such
that\equa{mu-l====}{\mbox{$\nu:=\l^{(1)}-\sum_{\b\in S_0}\b-2p_0\d$
is regular, and $\nu=w_0\cdot\l^{(0)}$.}} This implies that
$\tilde\nu_i-\tilde\l_i=0,\pm1$ for $1\le i\le m$. First assume
$\tilde\nu_i=\tilde\l_i-1$ for some $i\ne m$. Then (\ref{mu-l====})
implies that \equan{000}{\mbox{$\tilde\l_{i+1}=\tilde\l_i-1$,
$\tilde\nu_{i+1}=\tilde\l_i$, and $\d+\es_{i},\d-\es_{i+1}\in S_0$
but $\d-\es_i,\d+\es_{i+1}\notin S_0$.}} If we choose the subset
$S_1=(S\cup\{\d-\es_i,\d+\es_{i+1}\})\bs\{\d+\es_{i},\d-\es_{i+1}\}$
of $\D_1^+$, and also $w_1=(i,i+1)w_0\in W_0$, where $(i,i+1)$ is
the permutation exchanging $i$-th and $(i+1)$-th coordinates, then
(\ref{b-l-mu=}) has another term $\sign({w_1})=-\sign({w_0})$ which
cancels $\sign({w_0})$. Next assume
$\tilde\nu_m=\tilde\l_m-\frac12=-\frac12$. Then $\d+\es_m\in S_0$,
$\d-\es_m\notin S_0$. If we take $w_1=(-m)w$ (where $(-m)$ is the
element in $W_0$ which changes the sign of the $m$-th coordinate),
and\equan{0a0a}{\begin{array}{lll}
S_1=(S_0\cup\{\d\})\bs\{\d+\es_m\}&\mbox{if  }\d\notin S_0,\mbox{ \
or}\\[5pt]
S_1=S_0\bs\{\d,\d+\es_m\}\mbox{ and }p_1=p_0+1&\mbox{if }\d\in
S_0,\end{array}} then again (\ref{b-l-mu=}) has another term
$\sign({w_1})=-\sign({w_0})$ which cancels $\sign({w_0})$. Thus we
obtain (\ref{b-l-1-l-0=0}), and $\l^{(0)}\notin
P(V_{\l^{(1)}})\subset\{\l^{(1)},\l^{(0)},\l^{(-1)}\}$ by Lemma
\ref{lemma-possible-primitive}(1). Since $L_{\l^{(1)}}$ is finite
dimensional, we must have $\l^{(-1)}\in P(V_{\l^{(1)}})$. Using
Lemma \ref{lemma-tail-primitive}(1) and as in the proof of Lemma
\ref{lemma-tail-primitive}, we can obtain \equa{ch-l-1}{\ch
L_{\l^{(1)}}=\chi^{\rm BL}_{\l^{(1)}}-\ch
L_{\l^{(0)}}-(a_{\l^{(1)},\l^{(-1)}}-1)\ch L_{\l^{(-1)}}.} Since all
terms except the last term in the right-hand side of (\ref{ch-l-1})
are $W$-invariant, we obtain $a_{\l^{(1)},\l^{(-1)}}=1$, and we have
the second graph of (\ref{Stru-Kac-mod}).

Now we prove $\l^{(i)}\rrar\l^{(i-1)}$ for $i\ge2$. This is true
when $i\gg0$ by Lemma \ref{lemma-two-nontail} and Definition
\ref{defi-aty-type-set}(\ref{l-(i)}). Suppose $\l^{(i_0-1)}\notin
P(V_{\l^{(i_0)}})$ for some $i_0\ge2$ and $i_0$ is maximal. This
together with $b_{\l^{(i_0)},\l^{(i_0-1)}}=1$ (which can be proved
as in (\ref{b-l-1-l-0=0})) implies $\l^{(i_0-1)}\in
P_0(L_{\l^{(i_0)}})$ (cf.~(\ref{P-0-V-l})). By the choice of $i_0$,
we have $\l^{(i_0)}\in P(V_{\l^{(i_0+1)}})$, thus
$P_0(L_{\l^{(i_0)}})\subset P_0(V_{\l^{(i_0+1)}})$, and so
$\l^{(i_0-1)}\in P_0(V_{\l^{(i_0+1)}})$, which contradicts Lemma
\ref{lemma-possible-primitive}(\ref{P-0-subset}).

As in the proof of (\ref{b-l-1-l-0=0}), we have
$b_{\l^{(2)},\l^{(0)}}=1$.
 This together with
the arguments in the last paragraph shows $\l^{(2)}\rrar\l^{(0)}$.
Since $L_{\l^{(0)}},\,L_{\l^{(1)}}$ are finite-dimensional and
generalised Verma modules do not contains finite-dimensional
submodules, we must have $\l^{(1)}\rrar\mu$, $\l^{(0)}\rrar\nu$ in
$V_{\l^{(2)}}$ for some non-dominant weights $\mu,\nu$. But
$\l^{(1)}\rrar\mu$, $\l^{(0)}\rrar\nu$ must be quotients of
$V_{\l^{(1)}},\,V_{\l^{(0)}}$ respectively, by the first two graph
of (\ref{Stru-Kac-mod}), $\mu,\,\nu$ have to be $\l^{(-1)}$. Thus we
have the third graph of (\ref{Stru-Kac-mod}) except the relation
concerning $\l^{(-2)}$. If we do not have $\l^{(2)}\rrar\l^{(-2)}$
in $V_{\l^{(2)}}$, then in $V_{\l^{(3)}}$, the submodule
$U_{\l^{(2)}}$ generated by $v_{\l^{(2)}}$ (the primitive vector
with weight $\l^{(2)}$), which is a quotient of $V_{\l^{(2)}}$, has
to be $L_{\l^{(2)}}$, contradicting that $V_{\l^{(3)}}$ does not
contain a finite-dimensional submodule. Thus we have
$\l^{(2)}\rrar\l^{(-2)}$ and do not have $\l^{(-1)}\rrar\l^{(-2)}$
in $V_{\l^{(2)}}$. As in (\ref{ch-l-1}), we have
\equan{ch-l-2}{\ch L_{\l^{(2)}}=\chi^{\rm BL}_{\l^{(2)}}
-(a_{\l^{(2)},\l^{(-1)}}-1)\ch L_{\l^{(-1)}}
-(a_{\l^{(2)},\l^{(-2)}}-1)\ch L_{\l^{(-2)}},}
which proves that
$a_{\l^{(2)},\l^{(-1)}}=a_{\l^{(2)},\l^{(-2)}}=1$.

The fact that $\l^{(-2)}\rrar\l^{(-1)}$ is not needed for our
purpose of computing characters, nevertheless, we can prove this
fact for the case $\fg={\mathfrak{osp}}_{3|2}$ as follows:
 Let $\l=\l^{(2)}$, $\mu=\l-2(\l_0+1)\d$
(note that $\mu\notin P^{0+}_{\bar\l}$, i.e., $\bar\l$ is not the
atypical type of $\mu$), then as in the proof of
(\ref{b-l-1-l-0=0}), $b_{\l,\mu}=1$ and the unique $\fg_0$-highest
weight vector with weight $\mu$ (up to a nonzero scalar) is
$v_\mu:=f_{2\d}^{\l_0+1}v_\l$ (cf.~definition of Kac-modules in
(\ref{Kac-module})). We can also prove $b_{\l^{(-2)},\mu}=1$ as in
(\ref{b-l-1-l-0=0}), and $b_{\l^{(-3)},\mu}=0$ (in fact
$\l^{(-3)}<\mu$). Thus $\mu$ is a $\fg_0$-highest weight of
$L_{\l^{(-2)}}$, and $v_\mu$ is in the submodule $U_{\l^{(-2)}}$ of
$V_{\l^{(2)}}$ generated by $v_{\l^{(-2)}}$ (primitive vector with
weight $\l^{(-2)}$). Since the Kac-module $K_{\l^{(2)}}$ is finite
dimensional and $L_{\l^{(-1)}}$ is infinite-dimensional, $v_\mu$
must generate $v_{\l^{(-1)}}$ (the $\fg$-primitive vector with
weight $\l^{(-1)}$ in $V_{\l^{(2)}}$). This proves the third graph
of (\ref{Stru-Kac-mod}).

Now suppose $i\ge3$ and inductively assume that we have the graph in
(\ref{Stru-Kac-mod}) for all $\l^{(i_0)}$ with $i_0<i$. From this,
we can deduce as in (\ref{ch-l-1}) that
\equan{ch-l-i0}{\ch L_{\l^{(i_0)}}=\chi^{\rm BL}_{\l^{(i_0)}},\ \ \
2\le i_0<i.}
From Lemma \ref{lemma-possible-primitive}(1), we have
$P(V_{\l^{(i)}})\subset\{\l^{(i)},\l^{(i-1)},\l^{(2-i)},\l^{(1-i)},\l^{(-i)}\}$.
As above, we can then deduce
\begin{eqnarray*}\label{ch-l-i}
\ch L_{\l^{(i)}}\!\!&=\!\!&\chi^{\rm
BL}_{\l^{(i)}}-(a_{\l^{(i)},\l^{(1-i)}}-1)\ch L_{\l^{(1-i)}}
\nonumber\\&&-(a_{\l^{(i)},\l^{(-i)}}-1)\ch
L_{\l^{(-i)}}-a_{\l^{(i)},\l^{(2-i)}}\ch
L_{\l^{(2-i)}},
\end{eqnarray*}
which proves that
$a_{\l^{(i)},\l^{(1-i)}}=a_{\l^{(i)},\l^{(-i)}}=1$,
$a_{\l^{(i)},\l^{(2-i)}}=0$ as in (\ref{ch-l-1}), and we have the
last graph of (\ref{Stru-Kac-mod}).
\end{proof}

\subsection{Character and dimension formulae for irreducible modules}
Now we derive a character formula and a dimension formula for the
atypical finite-dimensional irreducible modules following similar
reasoning as that in the proof of \cite[Theorem 4.16]{SuZ2} (but the
case at hand is much easier). For atypical $\l\in P^+$, we define
\equan{m-l-m} {S_\l\!=\!\{\l,\l^{\si}\}\!\cap\!\{\nu\!\in\!
P^+\,|\,\nu\!\le\!\l\},\quad
m_\l\!=\!\#\big(\{\l,\l^\si\}\!\cap\!\{\nu\!\in\!
P^+\,|\,\nu\!\ge\!\l\}\big),} where $\l^\si$ is defined by
$(\ref{l=si})$. Then it is easy to see that
\[
S_\l=\left\{\begin{array}{l l}
\{\l,\l^\si\}, & \  \text{if $\l>\l^\si\in P^+$}, \\
\{\l\},& \  \text{otherwise},
\end{array}\right.
\quad
m_{\l}= \left\{\begin{array}{l l}
2, &\ \text{if $\mu<\mu^\si$}, \\
1, & \  \text{otherwise}.
\end{array}\right.
\]
For, $\mu\in S_\l$, we denote $\theta_{\l,\mu}\in W$ to be the
unique element with minimal length such that $\theta\cdot\l=\mu$,
namely, $\th_{\l,\mu}=1$ if $\l=\mu$ or $\th_{\l,\mu}=\si$
otherwise.

 By using Theorem \ref{theorem-verma-module},
 (\ref{ch-l-1}) and (\ref{ch-l-i}),  we immediately obtain
the following result.

\begin{theorem}\label{main-theo1}
Let $L_\l$ be the finite-dimensional irreducible $\fg$-module with
atypical highest weight $\l$.
\begin{enumerate}
\item The formal character $\ch L_\l$ of
$L_\l$ is given by
\equa{char-l-all}{\ch L_\l=\sum\limits_{\mu\in
S_\l}{\dis\frac{(-1)^{|\th_{\l,\mu}|}}{m_{\mu}R_{\bar0}}}\sum\limits_{w\in
W}\sign(w)w\Big(e^{\mu+\rho_{\bar0}}\prod\limits_{\b\in\D_1^+\bs\{\g_\mu\}}(1+e^{-\b})\Big),
}
where $\g_\mu$ is the atypical root of $\mu\in S_\l$, and
$|\th_{\l,\mu}|$ is the length of $\th_{\l,\mu}$.
\item
The dimension ${\rm dim\,} L_\l$ of $L_\l$ is given by
\equa{dim-l-all}{{\rm dim\,} L_\l=\sum\limits_{\mu\in
S_\l,\,B\subset\D_1^+\bs\{\g_\mu\}}(-1)^{|\th_{\l,\mu}|}m_{\mu}^{-1}\prod\limits_{\a\in
\D_{\bar0}^+}\dis\frac{(\a,\rho_{\bar0}+\mu-\sum_{\b\in
B}\b)}{(\a,\rho_{\bar0})}.}
\end{enumerate}
\end{theorem}

The following result on finite dimensional Kac modules can be easily
proven.
\begin{proposition}\label{Kac-ch}
The formal character $\ch K_\l$ of the finite-dimensional $($typical
or atypical$)$ Kac $\fg$-module $K_\l$ is $\ch K_\l=\chi^V_\l$ with
the right-hand side given by $(\ref{typical-char})$ unless
$\l^\si\in P^+$. If $\l^\si\in P^+$, then $\ch K_\l=\ch L_\l$ with
the right-hand side given by $(\ref{char-l-all})$. In particular,
\equa{Kac---}{\chi^V_{\l^{(i)}}=\ch L_{\l^{(i)}}+(-1)^{\d_{i,1}}\ch
L_{\l^{(i-1)}}\mbox{ \ for \ }i\ge1.}
\end{proposition}
\begin{proof}
Note that the graph of  $K_\l$ is obtained from that of $V_\l$ by
deleting all weight $\l^{(i)}$ with $i<0$. This together with
Theorems \ref{theorem-verma-module} and \ref{main-theo1}(1) implies
the result.
\end{proof}

An application of the character formula (\ref{char-l-all}) will be
found in Lemma \ref{ch-appl}.

\section{First and second cohomology groups}

In this section we apply results obtained on generalised Verma
modules and irreducible modules to determine the first and second
cohomology groups of ${\mathfrak{osp}}_{k|2}$ with coefficients in
finite-dimensional irreducible modules.

\subsection{Lie superalgebra cohomology}
Let us begin by recalling some basic concepts of Lie superalgebra
cohomology. The material can be found in many sources, say,
\cite{FL, SZ98, SuZ1}. For $p\ge 1$ and a finite-dimensional
$\gl$-module $V$, let the {\it space $C^p(\gl,V)$ of $p$-cochains}
be the $\Z_2$-graded vector space of all $p$-linear maps $\vp:
\gl\times\cdots\times\gl\to V$ satisfying the {\it super skew
symmetry condition}
 $$
\vp(x_1,...,x_i,x_{i+1},...,x_p)=-(-1)^{[x_i][x_{i+1}]}
\vp(x_1,...,x_{i+1},x_i,...,x_p) $$ for $1\le i\le p-1$, where,
 $[x_i]\in\Z_2$ denotes the parity of the element
$x_i$. Set $C^0(\gl,V)=V$. We define the {\it differential operator}
$d:C^p(\gl,V)\to C^{p+1}(\gl,V)$ by
\begin{eqnarray*}
\label{differential} &&\!\!\!\!\!\!\!\!(d\vp)(x_0,...,x_p)
\nonumber\\
&&=
\mbox{$\sum\limits_{i=0}^p$}(-1)^{i+[x_i]([\vp]+[x_0]+\cdots+[x_{i-1}])}x_i
\vp(x_0,...,\hat x_i,...,x_p) \nonumber\\&&
+\mbox{$\sum\limits_{i<j}$}(-1)^{j+[x_j]([x_{i+1}]+\cdots+[x_{j-1}])}
\vp(x_0,...,x_{i-1},[x_i,x_j],x_{i+1},...,\hat x_j,...,x_p),
\end{eqnarray*}
for $\vp\in C^p(\gl,V)$ and $x_0,...,x_p\in\gl$, where the sign
$\hat{\ }$ means that the element under it is omitted. It can be
verified that $d^2=0$. Set
\begin{eqnarray*}
Z^p(\gl,V)&=&\ker{d|_{C^p(\gl,V)}},\\
B^p(\gl,V)&=&\Im{d|_{C^{p-1}(\gl,V)}},\\
H^p(\gl,V)&=&Z^p(\gl,V)/B^p(\gl,V).
\end{eqnarray*}
The space $H^p(\gl,V)$ is the {\it $p$-th Lie superalgebra
cohomology group} of $\fg$ with coefficients in the module $V$.

We briefly discuss the long exact sequence of cohomology groups,
which is one of the essential tools used in this section. Let
$U,V,W$ be three $\gl$-modules such that
\equan{short-exact} { 0\to
U\rb{2pt}{\mbox{$\ ^{\ f}_{\dis\to}\ $}} V\rb{2pt} {\mbox{$\ ^{\
g}_{\dis\to}\ $}}W\to0 }
is a short exact sequence, where $f,g$ are
homogenous $\gl$-module homomorphisms. Then there exists a long
exact sequence
\equa{long-exact} { \cdots\to
H^p(\gl,U)\rb{3pt}{\mbox{$\,\ ^{\ f^p}_{\dis\!\!\!-\!\!\!\to}\ $}}
H^p(\gl,V) \rb{3pt}{\mbox{$\ ^{\ g^p}_{\dis-\!\!\!-\!\!\!\to}\
$}}H^p(\gl,W) \rb{3pt}{\mbox{$\ ^{\ d^*}_{\dis-\!\!\!-\!\!\!\to}\
$}} H^{p+1}(\gl,U)\to\cdots, }
where the maps $f^p,\,g^p$ can easily
be defined from $f,\,g$, and $d^*$ is the connecting homomorphism
(cf.~\cite[(2.50)]{SZ98}).

\subsection{Computation of cohomology groups}
If the cohomology group $H^*(\fg,L_\l)\ne0$ for integral
$\fg$-dominant $\l$, then $\l$ and $0$ must have the same typical,
hence $\l=\L^{(i)}$ for some $j\ge0$, where
$\L^{(0)}=(0\,|\,0,...,0)$ and
\equa{l-i-now}{\L^{(i)}=(2m+i-1-2s\,|\,i-1,0,...,0)\ \ \,\mbox{ \
for }i\ge1.}
For any $\fg$-module $V$, 
using the proof of \cite[Lemma \ref{notation-PV}]{SuZ1}, we have
(cf.~Notation \ref{notation-PV})
\equa{H-1-not=0}{H^1(\fg,V)\ne0\ \
\ \ \ \Lra\ \ \ \ \ \exists\, M(V\llar \L^{(0)}),}
where $M(V\llar \mu)$ means an indecomposable
module $M$ which contains the submodule $V$ such that the quotient
$M/V$ is isomorphic to $L_{\mu}$. Also,
\equa{H-2===0}{H^2(\fg,V)\ne0\ \ \ \ \Longrightarrow\ \ \ \
\exists\, M( V\llar\mu)\mbox{ \ for some $\mu$ with
}H^1(\fg,L_\mu)\ne0.} From this and Theorem
\ref{theorem-verma-module} and \cite[Lemma 6.7]{SuZ1}, we
immediately obtain
\begin{theorem}\label{homology}
Let $\fg=\ops$ with $k>2$. Let $L_\l$ and $K_\l$ respectively denote
the finite-dimensional irreducible and Kac $\fg$-modules with
highest weight $\l$. Then
\begin{eqnarray}\label{1-coho}\!\!\!\!\!\!\!\!\!\!&\!\!\!\!\!\!\!\!\!\!&
H^1(\fg,L_\l)\cong\left\{\begin{array}{lll}\C&\mbox{if \
}\l=\L^{(2)},\\[4pt]0&\mbox{otherwise}.\end{array}\right.\\[7pt]
 \label{1-coho-Verma}\!\!\!\!\!\!\!\!\!\!&\!\!\!\!\!\!\!\!\!\!&
 H^1(\fg,K_\l)\cong\left\{\begin{array}{lll}\C&\mbox{if \
 }\l=\L^{(3)},\\[4pt]0&\mbox{otherwise}.\end{array}\right.\\[7pt]
 \label{2-coho}\!\!\!\!\!\!\!\!\!\!&\!\!\!\!\!\!\!\!\!\!&
 H^2(\fg,L_\l)\cong 
 \left\{\begin{array}{lll}\C&\mbox{if \
 }\l=\L^{(1)},\,\L^{(3)},
 \\[4pt]0&\mbox{otherwise}.\end{array}\right.
\\[7pt]
 \label{2-coho-Verma}\!\!\!\!\!\!\!\!\!\!&\!\!\!\!\!\!\!\!\!\!&
  H^2(\fg,K_\l)\cong\left\{\begin{array}{lll}\C&\mbox{if \
  }\l=\L^{(1)},\,\L^{(4)}.\\[4pt]0&\mbox{otherwise}.\end{array}\right.
\end{eqnarray}
\end{theorem}
\begin{remark}
The first and second cohomology groups of $\mathfrak{sl}_{m|n}$
and $\mathfrak{osp}_{2|2n}$ for all $m$ and $n$
were computed in \cite{SuZ2}, and those of $\mathfrak{osp}_{3|2}$
were determined in \cite{G}.
\end{remark}

Let us first consider the following lemma, the proof of which may be
considered as an application of the character formula
(\ref{char-l-all}). The lemma will be used in the proof of Theorem
\ref{homology}.
\begin{lemma}\label{ch-appl}
\begin{enumerate}\item A module
$M(\L^{(0)}\rar\L^{(2)}\rar\L^{(1)})$ does not exist.\item
$H^2(\fg,\C)=0$.
\end{enumerate}
\end{lemma}
\begin{proof}Let $\C v_0$
be the trivial $\fg_{\bar0}$-module. We define the induced module
$$\widetilde M:={\rm Ind}_{\fg_{\bar0}}^{\fg}\C
v_0=U(\fg)\otimes_{U(\fg_{\bar0})}\C v_0\cong
U(\fg_{\bar1})\otimes_{\C}\C v_0,$$ here $U(\gl_{\bar1})$ denotes
the skew-symmetry tensor space of $\gl_{\bar1}$. Obviously,
\equan{SSSSSa}{\dis\ch\widetilde
M=R_1^2=\frac{R_1}{R_{\bar0}}\mbox{$\sum\limits_{w\in W}$}
\sign(w)w\Big(e^{\rho_{\bar0}+\rho_1}\mbox{$\prod\limits_{\b\in\D_1^+}$}(1+e^{-\b})\Big)
=\mbox{$\sum\limits_{B\subset\D_1^+}$}\chi^V_{\l_B},} where
$\l_B=2\rho_1-\sum_{\b\in B}\b$. We only need to consider those
$B$'s such that $\l_B$ is regular. When $\l_B$ is regular, we take
the unique $w_B\in W$ such that $w_B(\l_B+\rho):=\mu_B+\rho$ is
$\fg$-dominant. Thus $\ch\widetilde
M=\sum_{B\subset\D_1^+}\sign(w_B)\chi^V_{\mu_B}$. We have (to see
how it works, one can take $k=3,4,5,6$ as examples)
$$\ch\widetilde
M=\chi^V_{\L^{(2)}}-\chi^V_{\L^{(1)}}+...,$$ where the omitted terms
are (signed) sum of some $\chi^V_\l$'s with $\l\ne\L^{(i)}$ for
$i\in\Z_+$. Thus $\widetilde M$ is decomposed into a direct sum of
two submodules \equa{W-M==}{\widetilde M=\widetilde
M_0\oplus\widetilde M_1,} such that $\widetilde M_0$ has character
$\ch\widetilde M_0=\chi^V_{\L^{(2)}}-\chi^V_{\L^{(1)}}=\ch
L_{\L^{(2)}}+2\,\ch L_{\L^{(0)}}$ (cf.~(\ref{Kac---})), and all
$L_{\L^{(i)}}$'s are not composition factors of $\widetilde M_1$. In
particular, $\widetilde M$ does not have composition factor
$L_{\L^{(1)}}$, this proves (1) (since a module
$M(\L^{(0)}\rar\L^{(2)}\rar\L^{(1)})$  must be a quotient of
$\widetilde M$). From (1), we  obtain (cf.~Remark \ref{rema6.2})
\equa{Not-esss}{M(\L^{(1)}\to\L^{(2)}\to\L^{(0)})\mbox{ \ does not
exist}.}

Furthermore, $\widetilde M_0$ as a module generated by $v'_0$ (where
$v'_0$ is the projection of $v_0$ onto $\widetilde M_0$ with respect
to decomposition (\ref{W-M==})) must be indecomposable, thus it has
to be the module
\equa{MTo2222}{M(\L^{(0)}\rar\L^{(2)}\rar\L^{(0)}).} Now, from the
short exact sequence $0\to L_{\L^{(0)}}\to N\to L_{\L^{(2)}}\to0$,
where $N=M(\L^{(2)}\rar\L^{(0)})$, we obtain the exact sequence
(cf~(\ref{long-exact})) \equa{2===0}{0=H^1(\fg,L_{\L^{(0)}})\to
H^1(\fg,N)\stackrel{\phi}{\to} H^1(\fg,L_{\L^{(2)}})\to
H^2(\fg,\C)\to H^2(\fg,N).} Note from (\ref{H-2===0}) and
(\ref{Not-esss}) that $H^2(\fg,N)=0,$ also (\ref{MTo2222}) and
(\ref{H-1-not=0}) show that $H^1(\fg,N)\ne0$, thus $\phi$ is a
bijection. So (\ref{2===0})  with (\ref{H-2===0}) proves
$H^2(\fg,\C)=0$.
\end{proof}

The following remark will be used in the proof of Theorem
\ref{homology}.

 \begin{remark}
\label{rema6.2}\rm
Let $P(V)$ be a primitive weight graph. The {\it dual primitive
weight graph} $P^*(V)$ is the graph obtained from $P(V)$ by
reversing the directions of all arrows and changing all weights to
their dual weights. Note that $P^*(V)=P(V^*)$, where $V^*$ denote
the dual module of $V$. If we change the action of $\gl$ on $P(V^*)$
by the automorphism $\omega\in{\rm Aut}(\gl)$
  which interchanges $\C e_\a$'s and
$\C f_\a$'s, then we obtain another module, called the {\it inverse
module of} $V$, with graph $\wt P(V)$ obtained from $P(V)$ by
reversing the directions of all arrows $($note that using the
automorphism $\omega$, the module $L^*_\mu$ becomes $L_\mu$ for all
$\mu).$ In particular, we have \equa{inverse} {
\exists\,M(\mu\rrar\nu)\,\Lra\,\exists\,M(\mu\llar\nu)\,\Lra
\,\exists\,M(\mu^*\rrar\nu^*)\,\Lra\,\exists\,M(\mu^*\llar\nu^*). }
\end{remark}

\begin{proof}[Proof of Theorem \ref{homology}]
Note that a module $M(\L^{(i)}\rrar\L^{(0)})$ with $i\ge0$
must be a highest weight module thus a quotient
of $V_{\L^{(i)}}$ and so $i=2$ by (\ref{Stru-Kac-mod}). Thus there
is a module $M(\L^{(i)}\llar\L^{(0)})$ if and only if $i=2$
(cf.~(\ref{inverse})). By (\ref{H-1-not=0}), we have (\ref{1-coho})
since a module with structure $\L^{(2)}\rrar\L^{(0)}$ is unique.

Using (\ref{H-1-not=0}) and (\ref{1-coho}), we obtain that
$H^1(\fg,K_\l)\ne0$ only if $K_\l$ contains a composition factor
$L_{\L^{(2)}}$, i.e., $\l=\L^{(2)},\L^{(3)}$. One can prove as in
the previous paragraph that $H^1(\fg,K_{\L^{(3)}})\cong\C$. Also,
Lemma \ref{ch-appl} shows that $H^1(\fg,K_{\L^{(2)}})=0$. This
proves (\ref{1-coho-Verma}).

Analogously, using (\ref{H-2===0}), we obtain that
$H^2(\fg,L_\l)\ne0$ implies $\l=\L^{(0)},\L^{(1)},\L^{(3)}$. Lemma
\ref{ch-appl}(2) shows  $H^2(\fg,L_{\L^{(0)}})=0$. For
$\l=\L^{(1)}$, we have the following short exact sequence (where
$M=M(\L^{(2)}\rrar\L^{(1)})$, which exists by (\ref{Stru-Kac-mod}))
\equa{SHORT====}{0\to
L_{\L^{(1)}}\to M\to L_{\L^{(2)}}\to0.}
By (\ref{long-exact}), we have the exact sequence
\equa{LONG===}{0=H^1(\fg,M)\to H^1(\fg,L_{\L^{(2)}})\to
H^2(\fg,L_{\L^{(1)}})\to H^2(\fg,M)=0,}
where the first equality follows from (\ref{H-1-not=0}) and that
both $M(\L^{(0)}\to\L^{(2)}\to\L^{(1)})$ and
$M(\L^{(2)}\to\L^{(1)}\llar\L^{(0)})$ do not exist. The last
equality of (\ref{LONG===}) follows from (\ref{H-2===0}),
(\ref{1-coho}) and that both $M(\L^{(2)}\to\L^{(2)}\to\L^{(1)})$ and
$M(\L^{(2)}\to\L^{(1)}\llar\L^{(2)})$ do not exist (one can prove
that a module with structure $\L^{(2)}\to\L^{(1)}\llar\L^{(2)}$ must
be decomposable).  Thus $H^2(\fg,L_{\L^{(1)}})\cong
H^1(\fg,L_{\L^{(2)}})\cong\C$ by (\ref{LONG===}). Similarly, we have
(\ref{2-coho}) for $\L^{(3)}$.

Using (\ref{H-2===0}), we obtain that  $H^2(\fg,K_\l)\ne0$ only if
$\l=\L^{(1)},\L^{(4)}$. Since $K_{\L^{(1)}}=L_{\L^{(1)}}$, we have
(\ref{2-coho-Verma}) for $\L^{(1)}$. To prove (\ref{2-coho-Verma})
for $\L^{(4)}$, note that there exists an indecomposable module
$M=M(\L^{(4)}\to\L^{(3)}\llar\L^{(2)})$ which can be obtained from
the direct sum of two modules $K_{\L^{(4)}}$ and
$M(\L^{(3)}\llar\L^{(2)})$ by factoring the submodule
$U(\fg)(v_{\L^{(3)}}+v'_{\L^{(3)}})$, where
$v_{\L^{(3)}},\,v'_{\L^{(3)}}$ are respectively primitive vectors
with weight $\L^{(3)}$ in the two said modules. Then for this new
defined $M$, we have (\ref{SHORT====}) with  $L_{\L^{(1)}}$ replaced
by $K_{\L^{(4)}}$, and the above arguments now show that we have
(\ref{2-coho-Verma}) for $\L^{(4)}$.
\end{proof}

\vskip10pt \noindent{\bf Acknowledgements}. We thank Shun-Jen Cheng
and Ngau Lam for informing us about reference \cite{CLW}. This work
is supported by the Australian Research Council and the National
Science Foundation of China (grant no.~10825101).

\end{document}